\documentclass[11pt, twoside, letter]{article}
\usepackage[margin = 1in]{geometry}
\usepackage{array}
\usepackage{subcaption}
\usepackage{amsthm}
\usepackage{amsmath}
\usepackage{amsfonts}
\usepackage{microtype}
\usepackage{graphicx}
\usepackage{comment}
\usepackage{booktabs} 
\newtheorem{theorem}{Theorem}[section]
\newtheorem{lemma}[theorem]{Lemma}
\newtheorem{proposition}[theorem]{Proposition}
\newtheorem{corollary}[theorem]{Corollary}
\theoremstyle{definition}
\newtheorem{definition}[theorem]{Definition}

\usepackage{color}
\usepackage{graphicx}
\usepackage{pgfplots}

\usepackage{hyperref}

\theoremstyle{remark}
\newtheorem{remark}[theorem]{Remark}

\newcommand{\diag}{\mathop{\mathbf{diag}}}
\newcommand{\abs}[1]{\ensuremath{\left|#1\right|}}
\newcommand{\norm}[2][]{\ensuremath{\left\Vert #2 \right\Vert_{#1}}}

\renewcommand{\vec}[1]{\mathbf{#1}}

\begin{document}

\title{First-order methods Almost Always Avoid Saddle points: The case of Vanishing Step-sizes}

\author{Ioannis Panageas\\SUTD\\ioannis@sutd.edu.sg
\and Georgios Piliouras\\SUTD\\georgios@sutd.edu.sg
\and Xiao Wang\\SUTD\\xiao\_wang@sutd.edu.sg
}


\date{}
\maketitle

\begin{abstract}
In a series of papers \cite{LSJR16, PP17, LPP}, it was established that some of the most commonly used first order methods almost surely (under random initializations) and with step-size being small enough, avoid strict saddle points, as long as the objective function $f$ is $C^2$ and has Lipschitz gradient.  The key observation was that first order methods can be studied from a dynamical systems perspective, in which instantiations of Center-Stable manifold theorem allow for a global analysis. The results of the aforementioned papers were limited to the case where the step-size $\alpha$ is constant, i.e., does not depend on time (and bounded from the inverse of the Lipschitz constant of the gradient of $f$). It remains an open question whether or not the results still hold when the step-size is time dependent and vanishes with time.

In this paper, we resolve this question on the affirmative for gradient descent, mirror descent, manifold descent and proximal point. The main technical challenge is that the induced (from each first order method) dynamical system is time non-homogeneous and the stable manifold theorem is not applicable in its classic form. By exploiting the dynamical systems structure of the aforementioned first order methods, we are able to prove a stable manifold theorem that is applicable to time non-homogeneous dynamical systems and generalize the results in \cite{LPP} for vanishing step-sizes.
\end{abstract}

\section{Introduction}
Non-convex optimization has been studied extensively the last years and has been one of the main focuses of Machine Learning community. The reason behind the interest of ML community is that in many applications of interest, one has to deal with the optimization of a non-convex landscape. One of the key obstacles of non-convex optimization is the saddle points (which can outnumber the local minima \cite{dauphin2014identifying,pascanu2014saddle,choromanska2014loss}) and avoiding them is a fundamental problem \cite{MJ}.

Recent progress \cite{ge2015escaping, LPP} has shown that under mild regularity assumptions on the objective function, first-order methods such as gradient descent can provably avoid the so-called \textit{strict} saddle points\footnote{These are saddle points where the Hessian of the objective admits at least one direction of negative curvature. Such property has been shown to hold in a wide range of objective functions, see \cite{ge2015escaping, sun2015complete2,sun2015complete1, ge2016matrix,ge2017no,bhojanapalli2016global} and references therein.}.

In particular, a unified theoretical framework is established in \cite{LPP} to analyze the asymptotic behavior of first-order optimization algorithms such as gradient descent, mirror descent, proximal point, coordinate descent and manifold descent. It is shown that under random initialization, the aforementioned methods avoid strict saddle points almost surely. The proof exploits a powerful theorem from the dynamical systems literature, the so-called Stable-manifold theorem (see supplementary material for a statement of this theorem). For example, given a $C^2$ (twice continuously differentiable function) $f$ with $L$-Lipschitz gradient, gradient descent method \[\vec{x}_{k+1} = g(\vec{x}_k) := \vec{x}_k - \alpha \nabla f(\vec{x}_k)\] avoids strict saddle points almost surely, under the assumption that the stepsize is \textit{constant} and $0<\alpha<\frac{1}{L}$. The crux of the proof in \cite{LPP} is the use Stable-manifold theorem for the \textit{time-homogeneous}\footnote{This means that $g$ does not depend on time.} dynamical system $\vec{x}_{k+1} = g(\vec{x}_k)$. Stable-manifold theorem implies that the dynamical system $g$ avoids its unstable fixed points and with the fact that the unstable fixed points of the dynamical system $g$ coincide with the strict saddles of $f$ the claim follows.

In many applications/algorithms however the stepsize is adaptive or vanishing/diminishing (meaning $\lim_{k} \alpha_k = 0,$ e.g., $\alpha_k = \frac{1}{k}$ or $\frac{1}{\sqrt{k}}$). Such applications include stochastic gradient descent (see \cite{understand} for analysis of SGD for convex functions), urn models and stochastic approximation \cite{pemantle1990nonconvergence},  gradient descent \cite{bubeck14}, online learning algorithms like multiplicative weights update \cite{Arora1} (which is an instantiation of Mirror Descent with entropy regularizer).  It is also important to note that the choice of the stepsize is really crucial in the aforementioned applications as changing the stepsize can change the convergence properties (transition from convergence to oscillations/chaos \cite{nar18, PPP17}) and/or the rate of convergence \cite{nar18}.

The proof in \cite{LPP} does not carry over when the stepsize depends on time step, because the Stable-manifold theorem is not applicable.
Yet it remains an open question whether the results of \cite{LPP} hold for vanishing step-sizes and the authors stated in \cite{LPP} as an open question. This work resolves this question in the affirmative. Our main result is stated below informally.

\begin{theorem}[Informal]\label{Main:Informal}
Gradient Descent, Mirror Descent, Proximal point and Manifold descent, with vanishing step-size $\alpha_k$ of order $\Omega\left(\frac{1}{k}\right)$ avoid the set of strict saddle points (isolated and non-isolated) almost surely under random initialization.
\end{theorem}

\paragraph*{Organization of the paper.}
The paper is organized as follows: In Section 2 we give important definitions for the rest of the paper, in Section 3 we provide intuition and technical overview of our results, in Section 4 we show a new Stable-manifold theorem that is applicable to a class of time non-homogeneous dynamical systems and finally in Section 5 we show how this new manifold theorem can be applied to Gradient descent, Mirror Descent, Proximal point and Manifold Descent. 

\paragraph*{Notation.}

Throughout this paper, we denote
  $\mathbb{N}$ the set of nonnegative integers and $\mathbb{R}$ the set of real numbers,
 $\norm{\cdot}$ the Euclidean norm,
 bolded $\vec{x}$ the vector,
 $\mathbb{B}(\vec{x},\delta)$ the open ball centering at $\vec{x}$ with radius $\delta$,
 $g(k,\vec{x})$ the update rule for optimization algorithms indexed by $k\in\mathbb{N}$, $\tilde{g}(m,n,\vec{x})$ the composition $g(m,...,g(n+1,g(n,\vec{x}))...)$ for $m\ge n$,
 $\nabla f$ the gradient of $f:\mathbb{R}^d\rightarrow\mathbb{R}$ and
 $\nabla^2f(\vec{x})$ the Hessian of $f$ at $\vec{x}$,
 $D_{\vec{x}}g(k,\vec{x})$ the differential with respect to variable $\vec{x}$,

\section{Preliminaries}\label{sec:prelims}
In this section we provide all necessary definitions that will be needed for the rest of the paper.
\begin{definition}[Time (non)-homogeneous] We call a dynamical system $x_{k+1} = g(x_k)$ as time homogeneous since $g$ does not on step $k$. Furthermore, we call a dynamical system $x_{k+1} = g(k,x_k)$ time non-homogeneous as $g$ depends on $k$.
\end{definition}
\begin{definition}[Critical point] Given a $C^2$ (twice continuously differentiable) function $f: \mathbb{X} \to \mathbb{R}$ where $\mathbb{X}$ is an open, convex subset of $\mathbb{R}^d$, the following definitions are provided for completeness.
\begin{enumerate}
\item A point $\vec{x}^*$ is a critical point of $f$ if $\nabla f(\vec{x}^*)=0$.
\item A critical point is a local minimum if there is a neighborhood $U$ around $\vec{x}^*$ such that $f(\vec{x}^*)\le f(\vec{x})$ for all $\vec{x}\in U$, and a local maximum if $f(\vec{x}^*)\ge f(\vec{x})$.
\item A critical point is a saddle point if for all neighborhoods $U$ around $\vec{x}^*$, there are $\vec{x},\vec{y}\in U$ such that $f(\vec{x})\le f(\vec{x}^*)\le f(\vec{y})$.
\item A critical point $\vec{x}^*$ is isolated if there is a neighborhood $U$ around $\vec{x}^*$, and $\vec{x}^*$ is the only critical point in $U$.
\end{enumerate}
\end{definition}
This paper will focus on saddle points that have directions of strictly negative curvature, that is the concept of strict saddle. We also assume that the Hessian matrix of the objective function $f$ is invertible. 
\begin{definition}[Strict Saddle]
A critical point $\vec{x}^*$ of $f$ is a strict saddle if $\lambda_{\min}(\nabla^2 f(\vec{x}^*))<0$ (minimum eigenvalue of the Hessian computed at the critical point is negative).
\end{definition}

Let $\mathcal{X}^*$ be the set of strict saddle points of function $f$ and we follow the Definition 2 of \cite{LPP} for the global stable set of $\mathcal{X}^*$.

\begin{definition}[Global Stable Set and fixed points]
Given a dynamical system (e.g., gradient descent $\vec{x}_{k+1} = \vec{x}_k - \alpha_k \nabla f(\vec{x}_k)$)
\begin{equation}\label{eq:dynamicalsystem}
\vec{x}_{k+1} = g(k,\vec{x}_k),
\end{equation} the global stable set $W^s(\mathcal{X}^*)$ of $\mathcal{X}^*$ is the set of initial conditions where the sequence $\vec{x}_k$ converges to a strict saddle. This is defined as:
\[
W^s(\mathcal{X}^*)=\{\vec{x}_0:\lim_{k\rightarrow\infty}\vec{x}_k \in\mathcal{X}^*\}.
\]
Moreover, $\vec{z}$ is called a fixed point of the system (\ref{eq:dynamicalsystem}) if $\vec{z} = g(k ,\vec{z})$ for all natural numbers $k$.
\end{definition}

\begin{definition}[Manifold]
A $C^k$-differentiable, $d$-dimensional manifold is a space topological $M$, together with a collection of charts $\{(U_{\alpha},\phi_{\alpha})\}$, where each $\phi_{\alpha}$ is a $C^k$-diffeomorphism from an open subset $U_{\alpha}\subset M$ to $\mathbb{R}^d$. The charts are compatible in the sense that, whenever $U_{\alpha}\cap U_{\beta}\ne\emptyset$, the transition map $\phi_{\alpha}\circ\phi_{\beta}^{-1}:\phi_{\beta}(U_{\beta}\cap U_{\alpha})\rightarrow\mathbb{R}^d$ is of $C^k$.
\end{definition}

\section{Intuition and Overview}
\label{sec:intuition}
In this section we will illustrate why gradient descent and related first-order methods do not converge to saddle points, even for time varying/vanishing step-sizes $\alpha_k$ of order $\Omega\left(\frac{1}{k}\right)$. Consider the case of a quadratic, $f(\vec{x})=\frac{1}{2} \vec{x}^T A \vec{x}$ where $A = \diag(\lambda_1,...,\lambda_d)$ is a $d \times d$, non-singular, diagonal matrix with at least a negative eigenvalue. Let $\lambda_1,..., \lambda_{j}$ be the positive eigenvalues of $A$ (the first $j$) and $\lambda_{j+1},..., \lambda_{d}$ be the non-positive ones. It is clear that $\vec{x}^*=\vec{0}$ is the unique critical point of function $f$ and the Hessian $\nabla^2 f$ is $A$ everywhere (and hence at the critical point). Moreover, it is clear that $\vec{x}^*$ is a strict saddle point (not a local minimum).

Gradient descent with step-size $\alpha_k$ (it holds that $\alpha_k \geq 0$ for all $k$ and $\lim_{k \to \infty} \alpha_k =0$) has the following form:
\[
	\vec{x}_{k+1} =\vec{x}_k + \alpha_k A \vec{x}_k = (I - \alpha_k A) \vec{x}_k.
\]
Assuming that $\vec{x}_0$ is the starting point, then it holds that $\vec{x}_{k+1} = \left(\prod_{t=0}^k (I - \alpha_{k-t} A)\right) \vec{x}_0$. We conclude that
\begin{equation}\label{eq:Intuition}
\vec{x}_{k+1} = \diag\left(\prod_{t=0}^k (1-\lambda_1\alpha_t),...,\prod_{t=0}^k (1-\lambda_n\alpha_t)\right)\vec{x}_0.
\end{equation}
We examine when it is true that $\lim_{k \to \infty} \vec{x}_k = \vec{x}^*$. It is clear that $\prod_{t=0}^{\infty} (1-\lambda\alpha_t) = e^{\sum_{t=0}^{\infty} \ln (1 - \lambda\alpha_t)},$
and has the same convergence properties as
\begin{equation}\label{eq:nolog}
e^{-\lambda \sum_{t=0}^{\infty} \alpha_t}.
\end{equation}
 For $\lambda >0$, the term (\ref{eq:nolog}) converges to zero if and only if $\sum_{t=0}^{\infty} \alpha_t = +\infty$  which is true if $\alpha_t$ is $\Omega\left(\frac{1}{t}\right)$. Moreover, for $\lambda=0$ it holds that the term (\ref{eq:nolog}) remains a constant (independently of the choice of stepsize $\alpha_k$) and for $\lambda<0$ it holds that the term (\ref{eq:nolog}) diverges for $\alpha_t$ to be $\Omega\left(\frac{1}{t}\right)$. Therefore, for $\alpha_k$ being $\Omega\left(\frac{1}{k}\right)$ we conclude that $\lim_{k \to \infty} x_k = \vec{0}$ whenever the initial point $\vec{x}_0$ satisfies $x_0^i = 0$ ($i$-th coordinate of $\vec{x}_0$) for $\lambda_i \geq 0$.

Hence, if  $\vec{x}_0 \in E_s:=\textrm{span}(e_1,\ldots, e_j)$\footnote{$\{e_1,...,e_d\}$ denote the classic orthogonal basis of $\mathbb{R}^d$.}, then $\vec{x}_t$ converges to the saddle point $x^*$ and if $\vec{x}_0$ has a component outside $E_s$ then gradient descent diverges.  For the example above, the global stable set of $\vec{x}^*$ is the subspace $E_s$ which is of measure zero since $E_s$ is not full dimensional.
\begin{remark}[$\alpha_k$ of order $O\left( \frac{1}{k^{1+\epsilon}}\right)$] In the case $\alpha_k$ is a sequence of stepsizes that converges to zero with a rate $\frac{1}{k^{1+\epsilon}}$ for any $\epsilon>0$ (example $\frac{1}{k^2}, \frac{1}{2^k}$ etc), then it holds that $\sum_{t=0}^{\infty} \alpha_k$ converges and hence in our example above we conclude that $\lim_{k \to \infty} \vec{x}_k$ exists, i.e., $\vec{x}_k$ converges but not necessarily to a critical point.
\end{remark}

\subsection{Technical Overview}
The challenging part of proving our main result is the lack of generic theory for time non-homogeneous dynamical systems, either for continuous or discrete time. Moreover, as far as gradient descent, mirror descent, etc are concerned, the corresponding dynamical system that needs to be analyzed is more complicated when the objective function is not quadratic and the analysis of previous subsection does not apply.

Suppose we are given a function $f$ that is $C^2$, and $\vec{0}$ is a saddle point of $f$. The Taylor expansion of the gradient descent in a neighborhood of $\vec{0}$ is as follows:
\begin{equation}\label{eq:TaylorofGD}
\vec{x}_{k+1}=(I-\alpha_k\nabla^2f(\vec{0}))\vec{x}_k+\eta(k,\vec{x}_k),
\end{equation}
where $\eta(k,\vec{0})=\vec{0}$ and $\eta(k,\vec{x})$ is of order $o(\norm{\vec{x}})$ around $\vec{0}$ for all naturals $k$.

Due to the error term $\eta(k,\vec{x}_k)$, the approach for quadratic functions does not imply the existence of the stable manifold. Inspired by the proof of Stable-manifold theorem for time homogeneous ODEs, we prove a Stable-manifold theorem for discrete time non-homogeneous dynamical system (\ref{eq:TaylorofGD}). In words, we prove the existence of a manifold $W^s$ that is not of full dimension (it has the same dimension as $E^s$, where $E^s$ denotes the subspace that is spanned by the eigenvectors with corresponding positive eigenvalues of matrix $\nabla^2f(\vec{0})$).

To show this, we derive the expression of (\ref{eq:Intuition}) for the general function $f$ to be:
\begin{equation}\label{eq:general}
\vec{x}_{k+1}=A\left(k,0\right)\vec{x}_0+\sum_{i=0}^{k}A\left(k,i+1\right)\eta\left(i,\vec{x}_i\right),
\end{equation}
where $A\left(m,n\right)=\left(I-\alpha_m\nabla^2f(\vec{0})\right)...\left(I-\alpha_n\nabla^2f(\vec{0})\right)$ for $m\ge n$, and $A\left(m,n\right)=I$ if $m<n$. Next, we generate a sequence $\{\vec{x}_k\}_{k\in\mathbb{N}}$ from (\ref{eq:general}) with an initial point $\vec{x}_0=(\vec{x}_0^+,\vec{x}_0^-)$, where $\vec{x}_0^+\in E^s$ and $\vec{x}_0^-\in E^u$. If this sequence converges to $\vec{0}$, the equation (\ref{eq:general}) induces an operator $T$ on the space of sequences converging to $\vec{0}$, and the sequence $\{\vec{x}_k\}_{k\in\mathbb{N}}$ is the fixed point of $T$. This is so called the Lyapunov-Perron method (see supplementary material for some brief overview of the method). By Banach fixed point theorem (see supplementary material for the statement of the theorem), it can be proved that the sequence $\{\vec{x}_k\}_{k\in\mathbb{N}}$ (as the fixed point of $T$) exists and is unique. Furthermore, this implies that there is a unique $\vec{x}_0^-$ corresponding to $\vec{x}_0^+$, i.e. there exists a well defined function $\varphi:E^s\rightarrow E^u$ such that $\vec{x}_0^-=\varphi(\vec{x}_0^+)$.


\section{Stable Manifold Theorem for Time Non-homogeneous Dynamical Systems}
We start this section by showing the main technical result of this paper. This is a new stable manifold theorem that works for time non-homogeneous dynamical systems and is used to prove our main result (Theorem \ref{Main:Informal}) for Gradient Descent, Mirror Descent, Proximal Point and Manifold Descent. The proof of this theorem exploits the structure of the aforementioned first-order methods as dynamical systems.

\begin{theorem}[A new stable manifold theorem]\label{SMT}
Let $H$ be an invertible $d\times d$ real diagonal matrix with at least one negative eigenvalue, i.e. $H=diag\{\lambda_1,...,\lambda_d\}$ with $\lambda_1\ge\lambda_2\ge...\lambda_s>0>\lambda_{s+1}\ge...\ge\lambda_d$ and assume $\lambda_d<0$.
 Let $\eta(k,\vec{x})$ be a continuously differentiable function such that $\eta(k,\vec{0}) = \vec{0}$ and for each $\epsilon>0$, there exists a neighborhood of $\vec{0}$ in which it holds
\begin{equation}\label{Lipschitz}
\norm{\eta\left(k,\vec{x}\right)-\eta\left(k,\vec{y}\right)} \le\alpha_k\epsilon\norm{\vec{x}-\vec{y}}, \textrm{ for all naturals }k.
\end{equation}
Let $\{\alpha_k\}_{k\in\mathbb{N}}$ be a sequence of positive real numbers of order $\Omega\left(\frac{1}{k}\right)$ that converges to zero. We define the time non-homogeneous dynamical system
\begin{equation}
\vec{x}_{k+1}=g(k,\vec{x}_k), \textrm{ where }g(k,\vec{x})=(I-\alpha_kH)\vec{x}+\eta(k,\vec{x}).
\end{equation}
Suppose that $E=E^s\oplus E^u$, where $E^s$ is the span of the eigenvectors corresponding to negative eigenvalues of $H$, and $E^u$ is the span of the eigenvectors corresponding to nonnegative eigenvalues of $H$. Then there exists a neighborhood $U$ of $\vec{0}$ and a $C^1$-manifold $V(\vec{0})$ in $U$ that is tangent to $E^s$ at $\vec{0}$, such that for all $\vec{x}_0\in V(\vec{0})$, $\lim_{k\rightarrow\infty} g(k,\vec{x}_k)=\vec{0}$. Moreover, $\bigcap_{k=0}^{\infty}\tilde{g}^{-1}(k,0,U)\subset V(\vec{0})$.
\end{theorem}

We can generalize Theorem \ref{SMT} to the case where matrix $H$ is diagonalizable and for any fixed point $\vec{x}^*$ (instead of $\vec{0}$, using a shifting argument). The statement is given below.
\begin{corollary}\label{SMT2}
Let $\{\alpha_k\}_{k\in\mathbb{N}}$ be a sequence of positive real numbers that converges to zero. Additionally, $\alpha_k\in\Omega\left(\frac{1}{k}\right)$. Let $g(k,\vec{x}):\mathbb{R}^d\rightarrow\mathbb{R}^d$ be $C^1$ maps for all $k\in\mathbb{N}$ and
\begin{equation}\label{eq:dyn2}
\vec{x}_{k+1}=g(k,\vec{x}_k)
\end{equation}
be a time non-homogeneous dynamical system. Assume $\vec{x}^*$ is a fixed point, i.e. $g(k,\vec{x}^*)=\vec{x}^*$ for all $k\in\mathbb{N}$. Suppose the Taylor expansion of $g(k,\vec{x})$ at $\vec{x}^*$ in some neighborhood of $\vec{x}^*$,
\begin{equation}\label{eq:dyn3}
g(k,\vec{x})=g(k,\vec{x}^*)+D_{\vec{x}}g(k,\vec{x}^*)(\vec{x}-\vec{x}^*)+\theta(k,\vec{x}), \textrm{ satisfies }
\end{equation}
\begin{enumerate}
\item $D_{\vec{x}}g(k,\vec{x}^*)=I-\alpha_kG,$  $G$ invertible, real diagonalizable with at least one negative eigenvalue;
\item For each $\epsilon>0$, there exists an open neighborhood centering at $\vec{x}^*$ of radius $\delta>0$, denoted as $\mathbb{B}(\vec{x}^*,\delta)$, such that
\begin{equation}\label{condition2}
\norm{\theta(k,\vec{u}_1)-\theta(k,\vec{u}_2)}\le\alpha_k\epsilon\norm{\vec{u}_1-\vec{u}_2}
\end{equation}
for all $k\in\mathbb{N}$ and all $\vec{u}_1,\vec{u}_2\in \mathbb{B}(\vec{x}^*,\delta)$.
\end{enumerate}
There exists a open neighborhood $U$ of $\vec{x}^*$ and a $C^1$-manifold $W(\vec{x}^*)$ in $U$, with codimension at least one, such that for $\vec{x}_0\in W(\vec{x}^*)$, $\lim_{k\rightarrow\infty}g(k,\vec{x}^*)=\vec{x}^*$. Moreover, $\bigcap^{\infty}_{k=0}\tilde{g}^{-1}(k,0,U)\subset W(\vec{x}^*)$.
\end{corollary}

\begin{proof}
Since $G$ is diagonalizable, there exists invertible matrix $Q$ such that $G=Q^{-1}HQ$, hence
$
QGQ^{-1}=H,
$
where $H=diag\{\lambda_1,...,\lambda_d\}$ (i.e., $H$ is a diagonal matrix with entries $\lambda_1,...,\lambda_d$). Consider the map $\vec{z}=\varphi(\vec{x})=Q(\vec{x}-\vec{x}^*)$. $\varphi$ induces a new dynamical system in terms of $\vec{z}$ as follows:
\[
Q^{-1}\vec{z}_{k+1}=(I-\alpha_kG)Q^{-1}\vec{z}_k+\theta(k,Q^{-1}\vec{z}_k+\vec{x}^*).
\]
Multiplying by $Q$ from the left on both sides, we have
\begin{align}
\vec{z}_{k+1}&=Q(I-\alpha_kG)Q^{-1}\vec{z}_k+Q\theta(k,Q^{-1}\vec{z}_k+\vec{x}^*)
=(I-\alpha_kH)\vec{z}_k+\hat{\theta}(k,\vec{z}_k),\label{eq:dyn4}
\end{align}
where $\hat{\theta}(k,\vec{z}_k)=Q\theta(k,Q^{-1}\vec{z}_k+\vec{x}^*)$. Denote
$
q(k,\vec{z})=(I-\alpha_kH)\vec{z}+\hat{\theta}(k,\vec{z})
$
the update rule given by equation (\ref{eq:dyn4}). In order to apply Theorem \ref{SMT}, we next verify that $\hat{\theta}(k,\cdot)$ satisfies the condition (\ref{Lipschitz}) in Theorem \ref{SMT} for all $k\in\mathbb{N}$. It is essentially to verify that given any $\epsilon>0$, there exists a $\delta'>0$, such that
\begin{equation}\label{Lipschitz2}
\norm{\hat{\theta}(k,\vec{w}_1)-\hat{\theta}(k,\vec{w}_2)}=\norm{Q\theta(k,Q^{-1}\vec{w}_1+\vec{x}^*)-Q\theta(k,Q^{-1}\vec{w}_2+\vec{x}^*)}\le\alpha_k\epsilon\norm{\vec{w}_1-\vec{w}_2}
\end{equation}
for all $\vec{w}_1,\vec{w}_2\in \mathbb{B}(0,\delta')$. Let's elaborate it.
According to (\ref{condition2}) of condition 2, for any given $\epsilon>0$, and then $\frac{\epsilon}{\norm{Q}\norm{Q^{-1}}}$ is also a small positive number, there exists a $\delta>0$ (w.r.t. $\frac{\epsilon}{\norm{Q}\norm{Q^{-1}}}$), such that
\[
\norm{\theta(k,\vec{u}_1)-\theta(k,\vec{u}_2)}\le\alpha_k\frac{\epsilon}{\norm{Q}\norm{Q^{-1}}}\norm{\vec{u}_1-\vec{u}_2}
\]
for all $\vec{u}_1,\vec{u}_2\in\mathbb{B}(\vec{x}^*,\delta)$. Denote $V=Q(\mathbb{B}(\vec{x}^*,\delta)-\vec{x}^*)$, i.e.
\[
V=\{\vec{w}\in\mathbb{R}^d: \vec{w}=Q(\vec{u}-\vec{x}^*)\ \ \text{for some}\ \ \vec{u}\in\mathbb{B}(\vec{x}^*,\delta)\},
\]
and it is easy to see that $\vec{0}\in V$. Since $Q(\vec{u}-\vec{x}^*)$ is a diffeomorphism (composition of a translation and a linear isomorphism) from the open ball $\mathbb{B}(\vec{x}^*,\delta)$ to $\mathbb{R}^d$, $V$ is an open neighborhood (not necessarily a ball) of $\vec{0}$. Therefore, there exists an open ball at $\vec{0}$ with radius $\delta'$, denoted as $\mathbb{B}(\vec{0},\delta')$, such that $\mathbb{B}(\vec{0},\delta')\subset V$. Next we show that $\mathbb{B}(\vec{0},\delta')$ satisfies the inequality (\ref{Lipschitz2}).
By the definition of $V$, for any $\vec{w}_1,\vec{w}_2\in\mathbb{B}(\vec{0},\delta')\subset V$, there exist $\vec{u}_1,\vec{u}_2\in\mathbb{B}(\vec{x}^*,\delta)$, such that
\begin{equation}
\vec{w}_1=Q(\vec{u}_1-\vec{x}^*) ,\ \ \vec{w}_2=Q(\vec{u}_1-\vec{x}^*),
\end{equation}
and the inverse transformation is given by
\[
\vec{u}_1=Q^{-1}\vec{w}_1+\vec{x}^*, \ \ \vec{u}_2=Q^{-1}\vec{w}_2+\vec{x}^*.
\]
 Plugging to inequality (\ref{Lipschitz2}), we have
\begin{align*}
\norm{\hat{\theta}(k,\vec{w}_1)-\hat{\theta}(k,\vec{w}_2)}&=\norm{Q\theta(k,Q^{-1}\vec{w}_1+\vec{x}^*)-Q\theta(k,Q^{-1}\vec{w}_2+\vec{x}^*)}
\\
&=\norm{Q\theta(k,\vec{u}_1)-Q\theta(k,\vec{u}_2)}
\\
&\le\norm{Q}\norm{\theta(k,\vec{u}_1)-\theta(k,\vec{u}_2)}
\\
&\le\norm{Q}\alpha_k\frac{\epsilon\norm{\vec{u}_1-\vec{u}_2}}{\norm{Q}\norm{Q^{-1}}}
\\
&=\norm{Q}\alpha_k\frac{\epsilon}{\norm{Q}\norm{Q^{-1}}}\norm{(Q^{-1}\vec{w}_1+\vec{x}^*)-(Q^{-1}\vec{w}_2+\vec{x}^*)}
\\
&\le\norm{Q}\alpha_k\frac{\epsilon}{\norm{Q}\norm{Q^{-1}}}\norm{Q^{-1}}\norm{\vec{w}_1-\vec{w}_2}
\\
&=\alpha_k\epsilon\norm{\vec{w}_1-\vec{w}_2}.
\end{align*}
Thus the verification is complete.
So as a consequence of Theorem \ref{SMT}, there exists a $C^1$-manifold $V(\vec{0})$ such that for all $\vec{z}_0\in V(\vec{0})$, $\lim_{k\rightarrow\infty}\tilde{q}(k,0,\vec{z}_0)=\vec{0}$.
For the neighborhood $\varphi^{-1}(\mathbb{B}(\vec{0},\delta'))$ of $\vec{x}^*$, denote $W(\vec{x}^*)$ the local stable set of dynamical system given by $g(k,\vec{x})$, i.e.,
\[
W(\vec{x}^*)=\{\vec{x}_0\in\varphi^{-1}(\mathbb{B}(\vec{0},\delta)):\lim_{k\rightarrow\infty}\tilde{g}(k,0,\vec{x}_0)=\vec{x}^*\}.
\]
We claim that $W(\vec{x}^*)\subset\varphi^{-1}(V(\vec{0}))$ and the proof is as follows:
\\
Suppose $\vec{x}_0\in W(\vec{x}^*)$, then the sequence $\{\vec{x}_k\}_{k\in\mathbb{N}}$ generated by
$
\vec{x}_{k+1}=g(k,\vec{x}_k)
$ with initial condition $\vec{x}_0$ converges to $\vec{x}^*$. The map $\varphi$ induces a sequence $\{\vec{z}_k\}_{k\in\mathbb{N}}$, where $\vec{z}_0=\varphi(\vec{x}_0)$ and
\begin{align}
\vec{z}_{k+1}=\varphi(\vec{x}_{k+1})&=\varphi\left(g(k,\vec{x}_k)\right)
\\
&=Q\left(\vec{x}^*+(I-\alpha_{k}G)(\vec{x}_k-\vec{x}^*)+\theta(k,\vec{x}_k)-\vec{x}^*\right)
\\
& \ \ \ (\text{since $\vec{x}_k=\varphi^{-1}(\vec{z}_k)=Q^{-1}\vec{z}_k+\vec{x}^*$})
\\
&=Q(I-\alpha_{k}G)Q^{-1}\vec{z}_k+Q\theta(k,Q^{-1}\vec{z}_k+\vec{x}^*)
\\
&=(I-\alpha_{k}H)\vec{z}_k+\hat{\theta}(k,\vec{z}_k).
\end{align}
Since $\vec{z}_k=\varphi(\vec{x}_k)$, and $\vec{x}_k\rightarrow\vec{x}^*$, we have that $\vec{z}_k\rightarrow 0$. This implies sequence $\vec{z}_k$ generated by
$
\vec{z}_{k+1}=q(k,\vec{z}_k)
$ with initial condition $\vec{z}_0$ converges to $\vec{0}$, meaning that $\vec{z}_0=\varphi(\vec{x}_0)\in V(\vec{x}^*)$. Therefore $W(\vec{x}^*)\subset\varphi^{-1}(V(\vec{0}))$.
Let $U=\varphi^{-1}(\mathbb{B}(\vec{0},\delta))$ and the proof is complete.
\end{proof}

We conclude this section by the following corollary which can be proved using standard arguments about separability of $\mathbb{R}^d$ (every open cover has a countable subcover). We denote $W^s(\mathcal{A}^*)$ the set of initial conditions so that the given dynamical system $g$ converges to a fixed point $\vec{x}^*$ such that matrix $D_{\vec{x}}g(k,\vec{x}^*)$ has an eigenvalue with absolute value greater than one for all $k$.

\begin{corollary}\label{measure 0}
Let $g(k,\vec{x}):\mathbb{R}^d\rightarrow \mathbb{R}^d$ be the mappings defined in Theorem \ref{SMT2}. Then $W^s(\mathcal{A}^*)$ has Lebesgue measure zero.
\end{corollary}

\begin{proof}
 For each $\vec{x}^*\in \mathcal{A}^*$, there is an associated open neighborhood $U_{\vec{x}^*}$ promised by the Corollary \ref{SMT2}. $\bigcup_{\vec{x}^*\in \mathcal{A}^*}U_{\vec{x}^*}$ forms an open cover, and since $\mathbb{R}^d$ (more generally, any manifold) is second-countable we can find a countable subcover, so that $\bigcup_{\vec{x}^*\in \mathcal{A}^*}U_{\vec{x}^*}=\bigcup_{i=1}^{\infty}U_{\vec{x}^*_i}$
\\
By the definition of global stable set, we have
\[
W^s(\mathcal{A}^*)=\{\vec{x}_0:\lim_{k\rightarrow\infty}\tilde{g}(k,0,\vec{x}_0)\in\mathcal{A}^*\}.
\]
Fix a point $\vec{x}_0\in W^s(\mathcal{A}^*)$. Since $\tilde{g}(k,0,\vec{x}_0)\rightarrow\vec{x}^*\in\mathcal{A}^*$, there exists some nonnegative integer $T$ and all $t\ge T$, such that
$
\tilde{g}(t,0,\vec{x}_0)\in\bigcup_{\vec{x}^*\in\mathcal{A}^*}U_{\vec{x}^*}=\bigcup_{i=1}^{\infty}U_{\vec{x}^*_i}.
$
So $\tilde{g}(t,0,\vec{x}_0)\in U_{\vec{x}^*_i}$ for some $\vec{x}^*_i\in\mathcal{A}^*$ and all $t\ge T$. This is equivalent to
$
\tilde{g}(T+k,T,\tilde{g}(T,0,\vec{x}_0))\in U_{\vec{x}^*_i}
$
for all $k\ge 0$, and this implies that
$
\tilde{g}(T,0,\vec{x}_0)\in\tilde{g}^{-1}(T+k,T,U_{\vec{x}^*_i})
$
for all $k\ge 0$. And then we have
$
\tilde{g}(T,0,\vec{x}_0)\in\bigcap_{k=0}^{\infty}\tilde{g}^{-1}(T+k,T,U_{\vec{x}^*_i}).
$
Denote $S_{i,T}:=\bigcap_{k=0}^{\infty}\tilde{g}^{-1}(T+k,T,U_{\vec{x}^*_i})$
and the above relation is equivalent to
$\vec{x}_0\in \tilde{g}^{-1}(T,0,S_{i,T})$.
Take the union for all nonnegative integers $T$, we have
$
\vec{x}_0\in\bigcup_{T=0}^{\infty}\tilde{g}^{-1}(T,0,S_{i,T}).
$
And union for all $i$ we obtain that
$
\vec{x}_0\in \bigcup_{i=1}^{\infty}\bigcup_{T=0}^{\infty}\tilde{g}^{-1}(T,0,S_{i,T})
$
implying that
$
W^s(\mathcal{A}^*)\subset\bigcup_{i=1}^{\infty}\bigcup_{T=0}^{\infty}\tilde{g}^{-1}(T,0,S_{i,T}).
$
Since $S_{i,T}\subset W_n(\vec{x}^*)$ from Corollary \ref{SMT2}, and $W_n(\vec{x}^*)$ has codimension at least 1. This implies that $S_{i,T}$ has measure 0. Since the image of set zero-measure set under diffeomorphism is of measure 0, and countable union of zero-measure sets is of measure 0, we obtain $W^s(\mathcal{A}^*)$ is of measure 0.
\end{proof}

\section{Applications}
In this section, we apply Theorem \ref{SMT} (or its corollary \ref{SMT2}) to the four of the most commonly used first-order methods and we prove that each one of them avoids strict saddle points even with vanishing stepsize $\alpha_k$ of order $\Omega\left(\frac{1}{k}\right)$. We also assume that the Hessian $\nabla^2 f(\vec{x}^*)$ is invertible.
\subsection{Gradient Descent}
Let $f(\vec{x}):\mathbb{R}^d\rightarrow\mathbb{R}$ be a real-valued $C^2$ function, and $g(k,\vec{x})=\vec{x}-\alpha_k\nabla f(\vec{x})$ be the update rule of gradient descent, where $\{\alpha_k\}_{k\in\mathbb{N}}$ is a sequence of positive real numbers. 
Then
\begin{equation}\label{eq:GD}
\vec{x}_{k+1}=\vec{x}_k-\alpha_k\nabla f(\vec{x}_k)
\end{equation}
is a time non-homogeneous dynamical system.

\begin{theorem}\label{measure 0:GD}
Let $\vec{x}_{k+1}=g(k,\vec{x}_k)$ be the gradient descent algorithm defined by equation \ref{eq:GD}, and $\{\alpha_k\}_{k\in\mathbb{N}}$ be a sequence of positive real numbers of order $\Omega\left(\frac{1}{k}\right)$ that converges to zero. Then the stable set of strict saddle points has Lebesgue measure zero.
\end{theorem}

\begin{proof}
We need to verify that the Taylor expansion of $g(k,\vec{x})$ at $\vec{x}^*$ satisfies the conditions of Corollary \ref{SMT2}. Condition 1 is obvious since the Hessian $\nabla^2 f(\vec{x}^*)$ is diagonalizable and has at least one negative eigenvalue. It suffice to verify condition 2. Consider the Taylor expansion of $g(k,\vec{x})$ in a neighborhood $U$ of $\vec{x}^*$:
\begin{align*}
g(k,\vec{x})&=g(k,\vec{x}^*)+D_{\vec{x}}g(k,\vec{x}^*)(\vec{x}-\vec{x}^*)+\theta(k,\vec{x})
\\
&=\vec{x}^*+(I-\alpha_k\nabla^2f(\vec{x}^*))(\vec{x}-\vec{x}^*)+\theta(k,\vec{x})
\\
&=\vec{x}-\alpha_k\nabla^2f(\vec{x}^*)(\vec{x}-\vec{x}^*)+\theta(k,\vec{x}).
\end{align*}
So we can write
\[
\theta(k,\vec{x})=g(k,\vec{x})-\vec{x}+\alpha_k\nabla^2f(\vec{x}^*)(\vec{x}-\vec{x}^*),
\]
 and then the differential of $\theta(k,\vec{x})$ with respect to $\vec{x}$ is
\[
D_{\vec{x}}\theta(k,\vec{x})=D_{\vec{x}}(g(k,\vec{x})-\vec{x})+\alpha_k\nabla^2f(\vec{x}^*)
=-\alpha_k\nabla^2f(\vec{x})+\alpha_k\nabla^2f(\vec{x}^*).\label{eq:D of theta}
\]
From the Fundamental Theorem of Calculus and chain rule for multivariable functions, we have
\[
\theta(k,\vec{x})-\theta(k,\vec{y})=\int_0^1\frac{d}{dt}\theta(k,t\vec{x}+(1-t)\vec{y})dt=\int_0^1D_{\vec{z}}\theta(k,\vec{z})\vert_{\vec{z}=t\vec{x}+(1-t)\vec{y}}\cdot (\vec{x}-\vec{y})dt.
\]
By the assumption of $f$ being $C^2$, we have that $\nabla^2f(\vec{x})$ is continuous everywhere. And then for any given $\epsilon>0$, there exists a open ball $\mathbb{B}(\vec{x}^*)$ centering at $\vec{x}^*$, such that 
$
\norm{\nabla^2f(\vec{x})-\nabla^2f(\vec{x}^*)}\le\epsilon
$
 for all $\vec{x}\in\mathbb{B}(\vec{x}^*)$.
 And this implies that 
 $
 \norm{D_{\vec{x}}\theta(k,\vec{x})}\le\alpha_k\epsilon
 $
  for all $\vec{x}\in\mathbb{B}(\vec{x}^*)$. Since $t\vec{x}+(1-t)\vec{y}\in\mathbb{B}(\vec{x}^*)$ if $\vec{x},\vec{y}\in\mathbb{B}(\vec{x}^*)$, we have that 
 $
 \norm{D_{\vec{z}}\theta(k,\vec{z})\vert_{\vec{z}=t\vec{x}+(1-t)\vec{y}}}\le\alpha_k\epsilon
 $
  for all $t\in [0,1]$. By Cauchy-Schwarz inequality, we have
\begin{align*}
\norm{\theta(k,\vec{x})-\theta(k,\vec{y})}&=\norm{\int_0^1D_{\vec{z}}\theta(k,\vec{z})\vert_{\vec{z}=t\vec{x}+(1-t)\vec{y}}\cdot (\vec{x}-\vec{y})dt}
\\
&\le\left(\int_0^1\norm{D_{\vec{z}}\theta(k,\vec{z})\vert_{\vec{z}=t\vec{x}+(1-t)\vec{y}}}dt\right)\cdot\norm{\vec{x}-\vec{y}} 
\\
&=\alpha_k\epsilon\norm{\vec{x}-\vec{y}},
\end{align*}
the verification completes. By Corollary \ref{SMT2} and Corollary \ref{measure 0}, we conclude that the stable set of strict saddle points has Lebesgue measure zero. \end{proof}

\subsection{Mirror Descent}
We consider mirror descent algorithm in this section. Let $\mathbf{D}$ be a convex open subset of $\mathbb{R}^d$, and $M=\mathbf{D}\cap\mathbf{A}$ for some affine space $\mathbf{A}$. Given a function $f:M\rightarrow\mathbb{R}$ and a mirror map $\Phi$, the mirror descent algorithm with vanishing step-size is defined as
\begin{equation}\label{eq:MD}
\vec{x}_{k+1}=g(k,\vec{x}_k):=h(\nabla\Phi(\vec{x}_k)-\alpha_k\nabla f(\vec{x}_k)),
\end{equation}
where $h(\vec{x})=\text{argmax}_{\vec{z}\in M} \langle \vec{z},\vec{x}\rangle-\Phi(\vec{z})$.
\begin{definition}[Mirror Map]
We say that $\Phi$ is a mirror map if it satisfies the following properties.
\begin{itemize}
\item $\Phi:\mathbf{D}\rightarrow\mathbb{R}$ is $C^2$ and strictly convex.
\item The gradient of $\Phi$ is surjective onto $\mathbb{R}^d$, that is $\nabla\Phi(D)=\mathbb{R}^d$.
\item $\nabla_R\Phi$ diverges on the relative boundary of $M$, that is $\lim_{x\rightarrow\partial M}||\nabla_R\Phi(x)||=\infty$.
\end{itemize}
\end{definition}

\begin{theorem}\label{measure 0:MD}
Let $\vec{x}_{k+1}=g(k,\vec{x}_k)$ be the mirror descent algorithm defined by equation (\ref{eq:MD}), and $\{\alpha_k\}_{k\in\mathbb{N}}$ be a sequence of positive real numbers of order $\Omega\left(\frac{1}{k}\right)$ that converges to zero. Then the stable set of strict saddle points has Lebesgue measure zero.
\end{theorem}

\begin{proof}
Let $U\subset\mathbb{R}^d$ be an open ball centering at $\vec{x}^*$,the Taylor expansion of $g(k,\vec{x})$ in $U\cap M$ is of the form
  \[
  g(k,\vec{x})=g(k,\vec{x}^*)+(Id_{T_{\vec{x}^*}M}-\alpha_k\nabla^2_{R}\Phi(\vec{x}^*)^{-1}\nabla^2_Rf(\vec{x}^*))(\vec{x}-\vec{x}^*)+\theta(k,\vec{x}).
  \]
The fact that $g(k,\vec{x})$ satisfies the condition 1 of Corollary \ref{SMT2} follows from the proof of Proposition 10, \cite{LPP}, i.e. $\nabla^2_{R}\Phi(\vec{x}^*)^{-1}\nabla^2_Rf(\vec{x}^*)$ is similar to a symmetric linear operator (so diagonalizable) with at least one negative eigenvalue.
\\ 
  Next, we verify that $g(k,\vec{x})$ satisfies the condition 2 of Corollary \ref{SMT2}. From the Taylor expansion, we have
  \[
  \theta(k,\vec{x})=g(k,\vec{x})-\vec{x}^*-(Id_{T_{\vec{x}^*}M}-\alpha_k\nabla^2_{R}\Phi(\vec{x}^*)^{-1}\nabla^2_Rf(\vec{x}^*))(\vec{x}-\vec{x}^*),
  \]
  and
  \[
  D_{\vec{x}}\theta(k,\vec{x})=-\alpha_k\nabla^2_{R}\Phi(\vec{x})^{-1}\nabla^2_Rf(\vec{x})+\alpha_k\nabla^2_{R}\Phi(\vec{x}^*)^{-1}\nabla^2_Rf(\vec{x}^*).
  \]
  By the continuity of $\nabla^2_Rf$ and $\nabla^2_{R}\Phi(\vec{x})^{-1}$, the same argument as the proof of Theorem \ref{measure 0:GD} implies that the condition 2 of Corollary \ref{SMT2} is satisfied. Combing Corollary \ref{SMT2} and Corollary \ref{measure 0}, we conclude that the stable set of saddle points has Lebesgue measure zero.  
\end{proof}

\subsection{Proximal Point}
The proximal point algorithm is given by the iteration
\begin{equation}\label{eq:PP}
\vec{x}_{k+1}=g(k,\vec{x}_k):=\text{arg}\min_{\vec{z}}f(\vec{z})+\frac{1}{2\alpha_k}\norm{\vec{x}_k-\vec{z}}^2.
\end{equation}

\begin{theorem}\label{measure 0:PP}
Let $\vec{x}_{k+1}=g(k,\vec{x}_k)$ be the proximal point algorithm defined by equation (\ref{eq:PP}), and $\{\alpha_k\}_{k\in\mathbb{N}}$ be a sequence of positive real numbers of order $\Omega\left(\frac{1}{k}\right)$ that converges to zero. Then the stable set of strict saddle points has Lebesgue measure zero.
\end{theorem}

\begin{proof}
Different from the other First-order methods, the results is not a direct consequence of Corollary \ref{measure 0}, but instead, we need to apply part of the proof of Theorem \ref{SMT}. It is still necessary to verify that the Taylor expansion of $g(k,\vec{x})$ at $\vec{x}^*$ satisfies condition 1 and 2 of Corollary \ref{SMT2}.
From the proof of Proposition 3, \cite{LPP}, $g(k,\vec{x})+\alpha_k\nabla f(g(k,\vec{x}))=\vec{x}$. By implicit differentiation, $Dg(k,\vec{x})+\alpha_k\nabla^2f(g(k,\vec{x}))Dg(k,\vec{x})=I$, and
\[
Dg(k,\vec{x})=(I+\alpha_k\nabla^2f(g(k,\vec{x})))^{-1}.
\]
At saddle point $\vec{x}^*$, $Dg(k,\vec{x}^*)=(I+\alpha_k\nabla^2f(g(k,\vec{x}^*)))^{-1}$ that is diagonalizable since $\nabla^2f(\vec{x}^*)$ is diagonalizable. Suppose under the matrix $Q$, $Q\nabla^2f(\vec{x}^*)Q^{-1}=H$ that is diagonal. Then 
\begin{align}
Q(I+\alpha_k\nabla^2f(\vec{x}^*))^{-1}Q^{-1}&=\left(Q(I+\alpha_k\nabla^2f(\vec{x}^*))Q^{-1}\right)^{-1}
\\
&=\left(I+\alpha_kQ\nabla^2f(\vec{x}^*)Q^{-1}\right)^{-1}
\\
&=\left(I+\alpha_kH\right)^{-1}
\\
&=\text{diag}\{\frac{1}{1+\alpha_k\lambda_i}\},
\end{align}
where $\lambda_i$ are the eigenvalues of $H$. Notice that $
\frac{1}{1+\alpha_k\lambda_i}=1-\frac{\alpha_k\lambda_i}{1+\alpha_k\lambda_i}
$, the stable-unstable decomposition in the proof of Corollary \ref{SMT} holds. Furthermore, since $\alpha_k\in\Omega\left(\frac{1}{k}\right)$, $\frac{\alpha_k\lambda_i}{1+\alpha_k\lambda_i}$ is also of $\Omega\left(\frac{1}{k}\right)$. To see this, we can assume 
$
\alpha_k\lambda_i\ge\frac{1}{k-1}=\frac{1}{k}\cdot\frac{k}{k-1}
$, 
and then $\frac{k-1}{k}\alpha_k\lambda_i\ge\frac{1}{k}$ or $\left(1-\frac{1}{k}\right)\alpha_k\lambda_i\ge\frac{1}{k}$, and thus $\alpha_k\lambda_i\ge\frac{1}{k}(1+\alpha_k\lambda_i)$, implying that $\frac{\alpha_k\lambda_i}{1+\alpha_k\lambda_i}\ge\frac{1}{k}$. So the proof for Lemma \ref{boundedC} and Lemma \ref{boundedB} holds for the existence of stable manifold of proximal point algorithm if condition 2 of Corollary \ref{SMT2} is satisfied. To verify this, we consider the Taylor expansion of $g(k,\vec{x})$ at $\vec{x}^*$ has the form of
\begin{align*}
g(k,\vec{x})&=g(k,\vec{x}^*)+D_{\vec{x}}g(k,\vec{x}^*)(\vec{x}-\vec{x}^*)+\theta(k,\vec{x})
\\
&=\vec{x}^*+(I+\alpha_k\nabla^2f(g(k,\vec{x}^*)))^{-1}(\vec{x}-\vec{x}^*)+\theta(k,\vec{x}),
\end{align*}
and thus
\[
\theta(k,\vec{x})=g(k,\vec{x})-\vec{x}^*-(I+\alpha_k\nabla^2f(g(k,\vec{x}^*)))^{-1}(\vec{x}-\vec{x}^*).
\]
So the differential
\[
D_{\vec{x}}\theta(k,\vec{x})=(I+\alpha_k\nabla^2f(g(k,\vec{x})))^{-1}-(I+\alpha_k\nabla^2f(g(k,\vec{x}^*)))^{-1}.
\]
Since $f$ is of $C^2$, $g(k,\vec{x})$ and $\nabla^2f(\vec{x})$ are continuous, and then $\norm{\theta(k,\vec{x})-\theta(k,\vec{y})}\le\alpha_k\epsilon\norm{\vec{x}-\vec{y}}$ follows from the same argument as the proof of Theorem \ref{measure 0:GD}. So the verification completes and by Corollary \ref{SMT2} and Corollary \ref{measure 0}, we conclude that the stable set of strict saddle points is of Lebesgue measure zero.
\end{proof}

\subsection{Manifold Gradient Descent}
Let $M$ be a submanifold of $\mathbb{R}^d$, and $T_{\vec{x}}M$ be the tangent space of $M$ at $\vec{x}$. $P_M$ and $P_{T_{\vec{x}}M}$ be the orthogonal projector onto $M$ and $T_{\vec{x}}M$ respectively. Assume that $f:M\rightarrow \mathbb{R}$ is extendable to neighborhood of $M$ and let $\bar{f}$ be a smooth extension of $f$ to $\mathbb{R}^d$. Suppose that the Riemannian metric on $M$ is induced by Euclidean metric of $\mathbb{R}^d$, then the Riemannian gradient $\nabla_Rf(\vec{x})$ is the projection of the gradient of $f(\vec{x})$ on $\mathbb{R}^d$, i.e. $\nabla_Rf(\vec{x})=P_{T_{\vec{x}}M}\nabla f(\vec{x})$. Then the manifold gradient descent algorithm is:
\begin{equation}\label{MGD}
\vec{x}_{k+1}=g(k,\vec{x}_k):=P_M(\vec{x}_k-\alpha_kP_{T_{\vec{x}_k}M}\nabla f(\vec{x}_k)).
\end{equation}

\begin{theorem}\label{measure 0:MGD}
Let $\vec{x}_{k+1}=g(k,\vec{x}_k)$ be the manifold gradient descent defined by equation (\ref{MGD}), and $\{\alpha_k\}_{k\in\mathbb{N}}$ be a sequence of positive real numbers of order $\Omega\left(\frac{1}{k}\right)$ that converges to zero. Then the stable set of strict saddle points has measure zero.
\end{theorem}

\begin{proof}
According to the proof of Proposition 8, \cite{LPP}, for $\vec{v}\in T_{\vec{x}^*}M$, 
\[
D_{\vec{x}}g(k,\vec{x}^*)\vec{v}=P_{T_{\vec{x}^*}M}\vec{v}-\alpha_kP_{T_{\vec{x}^*}M}D(P_{T_{\vec{x}^*}M}\nabla\bar{f})(\vec{x}^*)\vec{v}.
\]
Let $\vec{x}-\vec{x}^*\in T_{\vec{x}^*}M$, the Taylor expansion in the tangent space can be written as
\[
g(k,\vec{x})=g(k,\vec{x}^*)+P_{T_{\vec{x}^*}M}(\vec{x}-\vec{x}^*)-\alpha_kP_{T_{\vec{x}^*}M}D(P_{T_{\vec{x}^*}M}\nabla\bar{f})(\vec{x}^*)(\vec{x}-\vec{x}^*)+\theta(k,\vec{x}).
\]
Using equation 4, \cite{AMT}, $P_{T_{\vec{x}}M}D(P_{T_{\vec{x}}M}\nabla\bar{f})(\vec{x})=\nabla^2_R f(\vec{x})$, which is exactly the Riemannian Hessian, and thus it is diagonalizable. So this verifies the condition 1 of Corollary \ref{SMT2}. Furthermore, the Taylor expansion gives
\[
\theta(k,\vec{x})=g(k,\vec{x})-\vec{x}^*-P_{T_{\vec{x}^*}M}(\vec{x}-\vec{x}^*)+\alpha_kP_{T_{\vec{x}^*}M}D(P_{T_{\vec{x}^*}M}\nabla\bar{f})(\vec{x}^*)(\vec{x}-\vec{x}^*),
\]
and then
\begin{align*}
D_{\vec{x}}\theta(k,\vec{x})&=D_{\vec{x}}g(k,\vec{x})-P_{T_{\vec{x}^*}M}+\alpha_kP_{T_{\vec{x}^*}M}D(P_{T_{\vec{x}^*}M}\nabla\bar{f})(\vec{x}^*).
\end{align*}
The continuity of $\nabla^2f$ implies that for each $\epsilon>0$, there exist neighborhood of $\vec{x}^*$, such that $\norm{D_{\vec{x}}\theta(k,\vec{x})}\le\epsilon$. Apply the argument in the proof of Theorem \ref{measure 0:GD} (Fundamental Theorem of Calculus and Cauchy-Schwarz inequality), we can conclude that condition 2 of Corollary \ref{SMT2} is satisfied. then combing with Corollary \ref{measure 0}, we have that the stable set of strict saddle points has measure (induced by metric $R$) zero.
\end{proof}

For the case when $M$ is not a submanifold of $\mathbb{R}^d$, the manifold gradient descent algorithm depends on the Riemannian metric $R$ defined intrinsically, i.e., $R$ is not induced by any ambient metric.
Given $f:M\rightarrow\mathbb{R}$, the Riemannian gradient $\nabla_Rf$ is defined to be the unique vector field such that $R(\nabla_Rf,X)=\partial_Xf$ for all vector field $X$ on $M$. In local coordinate systems $\vec{x}(p)=(x_1,...,x_d)$, $p\in M$, the Riemannian gradient is written as
$\nabla_Rf(\vec{x})=\left(R^{1j}\frac{\partial f}{\partial x_j},..., R^{dj}\frac{\partial f}{\partial x_j}\right)=\left(R^{ij}\right)\cdot\nabla f(\vec{x}),
$
where $\left(R^{ij}\right)$ is the inverse of the metric matrix at the point $\vec{x}$ and $R^{ij}\frac{\partial f}{\partial x_j}=\sum_jR^{ij}\frac{\partial f}{\partial x_j}$ as the Einstein convention. Then the update rule (in a local coordinate system) is
\begin{equation}\label{eq:IMGD}
\vec{x}_{k+1}=g(k,\vec{x}_k):=\vec{x}_k-\alpha_k\left(R^{ij}\right)\cdot\nabla f(\vec{x}_k).
\end{equation}

\begin{theorem}\label{measure 0:IMGD}
Let $\vec{x}_{k+1}=g(k,\vec{x}_k)$ be the manifold gradient descent defined by equation (\ref{eq:IMGD}), and $\{\alpha_k\}_{k\in\mathbb{N}}$ be a sequence of positive real numbers of order $\Omega\left(\frac{1}{k}\right)$ that converges to zero. Then the stable set of strict saddle points has measure zero.
\end{theorem}

\begin{proof}
Let $\vec{x}^*\in\mathcal{X}^*$, then $\nabla f(\vec{x}^*)=0$, and $g(k,\vec{x}^*)=\vec{x}^*$. To show that $\vec{x}^*$ is unstable, consider the differential 
\[
D_{\vec{x}}g(k,\vec{x})=I-\alpha_kD_{\vec{x}}\left(\left(R^{ij}\right)\cdot\nabla f(\vec{x})\right),
\]
where 
\begin{align}
D_{\vec{x}}\left(\left(R^{ij}\right)\cdot\nabla f(\vec{x})\right)&=\left[
\begin{array}{lll}
\frac{\partial}{\partial x_1}( R^{1j}\frac{\partial f}{\partial x_j}) & \dotsm & \frac{\partial}{\partial x_d}( R^{1j}\frac{\partial f}{\partial x_j}) 
\\
\vdots &  & \vdots
\\
\frac{\partial}{\partial x_1}( R^{dj}\frac{\partial f}{\partial x_j}) & \dotsm & \frac{\partial}{\partial x_d}( R^{dj}\frac{\partial f}{\partial x_j})
\end{array}
\right]
\\
&=\left[
\begin{array}{lll}
\frac{\partial R^{1j}}{\partial x_1}\frac{\partial f}{\partial x_j}+R^{1j}\frac{\partial^2f}{\partial x_1\partial x_j} & \dotsm & \frac{\partial R^{1j}}{\partial x_d}\frac{\partial f}{\partial x_j}+R^{1j}\frac{\partial^2 f}{\partial x_d\partial x_j}
\\
\vdots & & \vdots
\\
\frac{\partial R^{dj}}{\partial x_1}\frac{\partial f}{\partial x_j}+R^{dj}\frac{\partial^2 f}{\partial x_1\partial x_j} & \dotsm & \frac{\partial R^{dj}}{\partial x_d}\frac{\partial f}{\partial x_j}+R^{dj}\frac{\partial^2 f}{\partial x_m\partial x_j}
\end{array}
\right]
\\
&=\left[
\begin{array}{lll}
\frac{\partial R^{1j}}{\partial x_1}\frac{\partial f}{\partial x_j} & \dotsm & \frac{\partial R^{1j}}{\partial x_d}\frac{\partial f}{\partial x_j}
\\
\vdots & & \vdots
\\
\frac{\partial R^{dj}}{\partial x_1}\frac{\partial f}{\partial x_j} & \dotsm & \frac{\partial R^{dj}}{\partial x_d}\frac{\partial f}{\partial x_j}
\end{array}
\right]
+
\left[
\begin{array}{lll}
R^{1j}\frac{\partial^2f}{\partial x_1\partial x_j} & \dotsm &R^{1j}\frac{\partial^2 f}{\partial x_d\partial x_j}
\\
\vdots & & \vdots
\\
R^{dj}\frac{\partial^2 f}{\partial x_1\partial x_j} & \dotsm & R^{dj}\frac{\partial^2 f}{\partial x_d\partial x_j}
\end{array}
\right].
\end{align}
Since at $\vec{x}^*$, $\nabla f(\vec{x}^*)=0$, i.e. $\frac{\partial f}{\partial x_j}=0$ for all $j$, we have
\[
D_{\vec{x}}\left(\left(R^{ij}\right)\cdot\nabla f(\vec{x}^*)\right)=\left[
\begin{array}{lll}
R^{1j}\frac{\partial^2f}{\partial x_1\partial x_j} & \dotsm &R^{1j}\frac{\partial^2 f}{\partial x_d\partial x_j}
\\
\vdots & & \vdots
\\
R^{dj}\frac{\partial^2 f}{\partial x_1\partial x_j} & \dotsm & R^{dj}\frac{\partial^2 f}{\partial x_d\partial x_j}
\end{array}
\right]_{\vec{x}=\vec{x}^*}=\left(R^{ij}\right)\cdot \left(\frac{\partial^2 f}{\partial x_i\partial x_j}\right)\Bigg|_{\vec{x}=\vec{x}^*}.
\]
Recall that $\left(R^{ij}\right)=\left(R_{ij}\right)^{-1}$, and as it is shown in \cite{LPP}, by the similar transformation under $\left(R_{ij}\right)^{\frac{1}{2}}$, we have 
\[
\left(R_{ij}\right)^{\frac{1}{2}}\cdot D_{\vec{x}}\left(\left(R^{ij}\right)\cdot\nabla f(\vec{x}^*)\right)\cdot\left(R_{ij}\right)^{-\frac{1}{2}}=\left(R_{ij}\right)^{-\frac{1}{2}}\cdot\left(\frac{\partial^2 f}{\partial x_i\partial x_j}\right)\cdot\left(R_{ij}\right)^{-\frac{1}{2}},
\]
that is a symmetric matrix, so it is diagonalizable. And thus, $D_{\vec{x}}\left(\left(R^{ij}\right)\cdot\nabla f(\vec{x}^*)\right)$ is diagonalizable and has the same eigenvalue with $\left(\frac{\partial^2 f}{\partial x_i\partial x_j}\right)$, meaning that it has negative eigenvalues. So the Riemmanian Hessian at $\vec{x}^*$ as at least one negative eigenvalue.  
\\
\indent
The Taylor expansion of $g(k,\vec{x})$ at $\vec{x}^*$ is
\[
g(k,\vec{x})=g(k,\vec{x}^*)+\left(I-\alpha_k\left(R^{ij}(\vec{x}^*)\right)\cdot\left(\frac{\partial^2 f}{\partial x_i\partial x_j}(\vec{x}^*)\right)\right)(\vec{x}-\vec{x}^*)+\theta(k,\vec{x}),
\]
where 
\[
\left(R^{ij}(\vec{x}^*)\right)\cdot\left(\frac{\partial^2 f}{\partial x_i\partial x_j}(\vec{x}^*)\right)=\left(R^{ij}\right)\cdot \left(\frac{\partial^2 f}{\partial x_i\partial x_j}\right)\Bigg|_{\vec{x}=\vec{x}^*}.
\]
We have 
\[
\theta(k,\vec{x})=g(k,\vec{x})-\vec{x}^*-\left(I-\alpha_k\left(R^{ij}(\vec{x}^*)\right)\cdot\left(\frac{\partial^2 f}{\partial x_i\partial x_j}(\vec{x}^*)\right)\right)(\vec{x}-\vec{x}^*)
\]
and 
\[
D_{\vec{x}}\theta(k,\vec{x})=-\alpha_k\left(R^{ij}(\vec{x})\right)\cdot\left(\frac{\partial^2 f}{\partial x_i\partial x_j}(\vec{x})\right)+\alpha_k\left(R^{ij}(\vec{x}^*)\right)\cdot\left(\frac{\partial^2 f}{\partial x_i\partial x_j}(\vec{x}^*)\right).
\]
By the continuity of $\left(\frac{\partial^2 f}{\partial x_i\partial x_j}(\vec{x})\right)$, the same argument as the verification of condition 2 in the proof of Theorem \ref{measure 0:GD} implies that $\theta(k,\vec{x})$ satisfies the condition 2 of Corollary \ref{SMT2}. Combining with Corollary \ref{measure 0}, we conclude that the stable set of strict saddle points has measure (induced by metric $R$) zero.
\end{proof}

\section{Conclusion}
In this paper, we generalize the results of \cite{LPP} for the case of vanishing stepsizes. We showed that if the stepsize $\alpha_k$ converges to zero with order $\Omega\left(\frac{1}{k}\right)$, then gradient descent, mirror descent, proximal point and manifold descent still avoid strict saddles. We believe that this is an important result that was missing from the literature since in practice vanishing or adaptive stepsizes are commonly used. Our main result boils down to the proof of a Stable-manifold theorem \ref{SMT} that works for time non-homogeneous dynamical systems and might be of independent interest. We leave as an open question the case of Block Coordinate Descent (as it also appears in \cite{LPP}).

\begin{figure}[h!]
\centering
\includegraphics[width=0.8\textwidth]{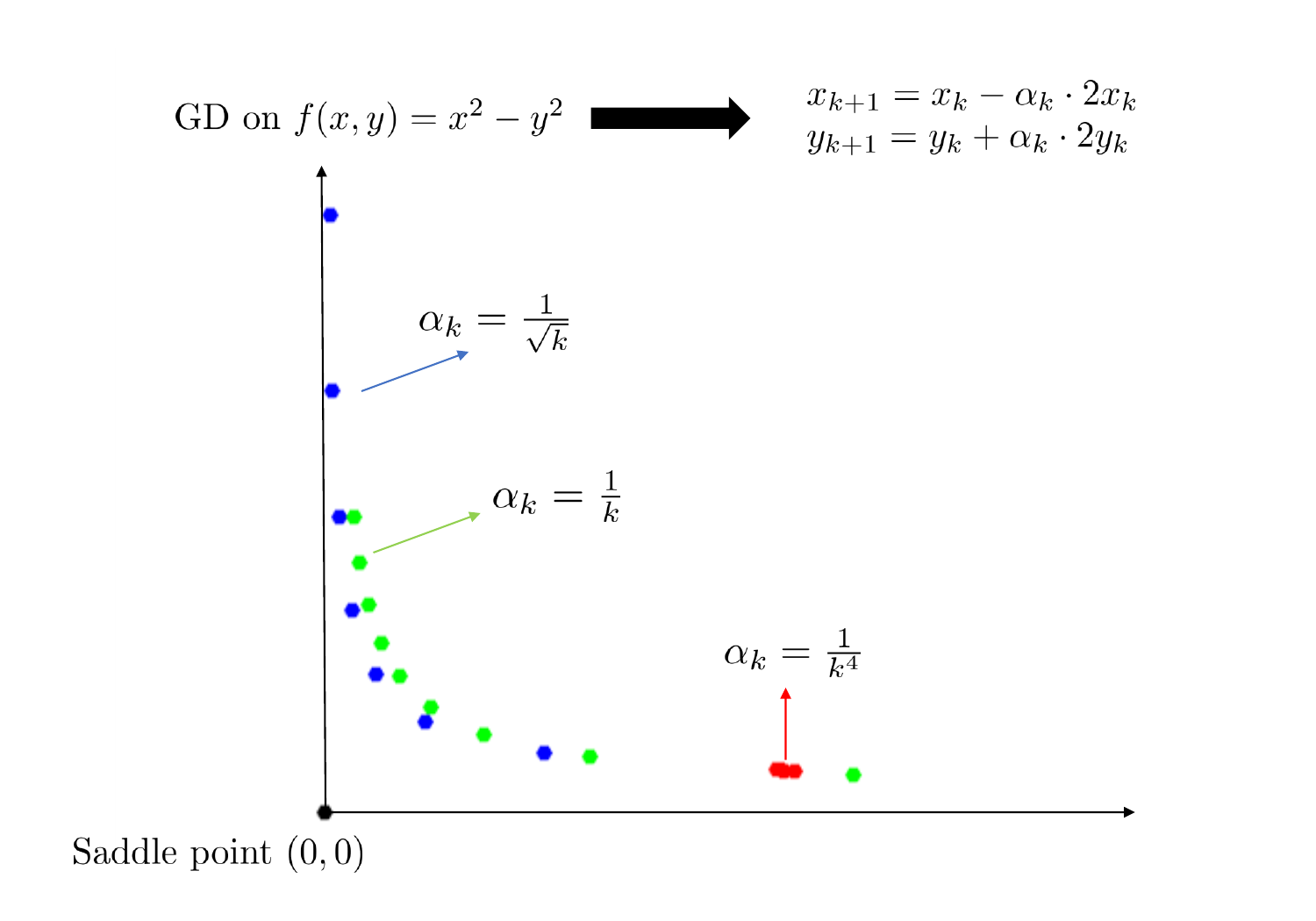}
\caption{Steps of Gradient Descent for $x^2-y^2$. $(0,0)$ is a strict saddle. Stepsizes $\frac{1}{\sqrt{k}}, \frac{1}{k}$ (blue, green)  avoid $(0,0)$ (blue faster than green). Stepsize $\frac{1}{k^4}$ (red) converges to a non-critical point.}
\label{fig:stepsize}
\end{figure}

\paragraph*{Acknowledgement}
We thank Nicolas Boumal and Andreea-Alexandra Musat for pointing out that the proof of the main theorem to work requires the extra assumption that the Hessian of $f$ is invertible.

\newpage
\bibliography{bibli}
\bibliographystyle{plain}
\newpage
\appendix
\section{Statement of theorems used and for completion}
\begin{theorem}[Banach Fixed Point Theorem, 2.1 \cite{AL}]
Let $(X,d)$ be a complete metric space, then each contraction map $T:X\rightarrow X$ has unique fixed point.
\end{theorem}
\begin{theorem}[Center-Stable Manifold Theorem, III.7  \cite{shub1987global}]\label{thm:manifold}
	Let $x^*$ be a fixed point for the $C^r$ local diffeomorphism $g: \mathcal{X} \to \mathcal{X}$. Suppose that $E = E_s \oplus E_u$, where $E_s$ is the span of the eigenvectors corresponding to eigenvalues of magnitude less than or equal to one of $D g(x^*)$, and $E_u$ is the span of the eigenvectors corresponding to eigenvalues of magnitude greater than one of $D g(x^*)$\footnote{Jacobian of function $g$.}. Then there exists a $C^r$ embedded disk $W^{cs}_{loc}$ of dimension $dim(E^s)$ that is tangent to $E_s$ at $x^*$ called the \emph{local stable center manifold}.  Moreover, there exists a neighborhood $B$ of $x^*$, such that $g(W^{cs}_{loc} ) \cap B \subset W^{cs} _{loc}$, and $\cap_{k=0}^\infty g^{-k} (B) \subset W^{cs}_{loc}$.
\end{theorem}

\section{Lyapunov-Perron Method}
The Lyapunov-Perron method has been developed by A.M. Lyapunov and O. Perron for the proof of the existence of stable and unstable manifolds of hyperbolic equilibrium points of ODEs. It uses the integral equation formulation of the differential equation and constructs the invariant manifold as a fixed point of an operator that is derived from this integral equation. The following case for time homogeneous ODEs can be found in Section 2.7, \cite{Perk}. Let $F:\mathbb{R}^d\rightarrow\mathbb{R}^d$ be of $C^1$ with $F(\vec{0})=\vec{0}$, consider the ODE 
\[
\frac{d\vec{x}}{dt}=F(\vec{x}),
\]
whose linear approximation at $\vec{0}$ is 
\[
\frac{d\vec{x}}{dt}=A\vec{x}+\eta(\vec{x}).
\]
By a change of coordinate system, $A$ is assumed to be decomposed to stable-unstable blocks respectively. Consider the operator $T$ defined as follows:
\[
Tu(t,\vec{x}_0)=U(t)\vec{x}_0+\int_0^tU(t-s)\eta(u(s,\vec{x}_0))ds-\int_t^{\infty}V(t-s)\eta(u(s,\vec{x}_0))ds
\]
where $\vec{x}_0$ is the initial point, $U(t)$ and $V(t)$ are integral operators from the block decomposition of $A$. The stable manifold is the fixed point of $T$ following from the Banach fixed point theorem.

\section{Proof of Theorem \ref{SMT}}

\begin{proof}
Denote $A\left(m,n\right)=\left(I-\alpha_mH\right)...\left(I-\alpha_nH\right)$ for $m\ge n$, and $A\left(m,n\right)=I$ if $m<n$. Then the dynamical system can be written as
\begin{equation}\label{eq:solution}
x_{k+1}=A\left(k,0\right)x_0+\sum_{i=0}^{k}A\left(k,i+1\right)\eta\left(i,x_i\right).
\end{equation}
Since $H$ is diagonal, the matrix $A\left(m,n\right)$ has the form of 
\[
\left(
\begin{array}{ll}
B(m,n) &
\\
&C(m,n)
\end{array}
\right)
\]
where $B_k$ and $C_k$ are diagonal as well and corresponding to \emph{stable} and \emph{unstable} subspaces of $I-\alpha_kH$ at $0$. Using the same notation of denoting  $A\left(m,n\right)$, we denote 
\[
B\left(m,n\right)=B_m\cdot...\cdot B_n
\]
and 
\[
C\left(m,n\right)=C_m\cdot...\cdot C_n.
\]
Let $v$ be a vector, we denote $v^+$ the stable component of $v$ and $v^-$ the unstable component of $v$. Then the solution (\ref{eq:solution}) can be written in terms of stable and unstable components as
\[
x_{k+1}^+=B\left(k,0\right)x_0^++\sum_{i=0}^kB\left(k,i+1\right)\eta^+\left(i,x_i\right)
\]
and
\[
x_{k+1}^-=C\left(k,0\right)x_0^-+\sum_{i=0}^kC\left(k,i+1\right)\eta^-\left(i,x_i\right).
\]
If $x_{k+1}\rightarrow 0$ as $k\rightarrow\infty$, then $x_{k+1}^-\rightarrow 0$ as $k\rightarrow\infty$. So we let $k\rightarrow\infty$, the following limit must holds:
\[
\lim_k\left(C\left(k,0\right)x_0^-+\sum_{i=0}^kC\left(k,i+1\right)\eta^-\left(i,x_i\right)\right)=0.
\]
Then we can solve $x_0^-$ in limit:
\[
x_0^-=\lim_k\big(C_0^{-1}\cdot...\cdot C_k^{-1}x_{k+1}^{-1}-\left[C_0^-\eta^-\left(0,x_0\right)
+\dotsm
+C_0^{-1}\cdot...\cdot C_k^{-1}\eta^-\left(k,x_k\right)\right]\big),
\]
and then by taking limit as $k\rightarrow\infty$,
\begin{equation}\label{eq:inftysum}
x_0^-=-\sum_{i=1}^{\infty}C(i-1,0)^{-1}\eta^-\left(i-1,x_{i-1}\right),
\end{equation}
where $C(m,n)^{-1}$ denotes the inverse of $C(m,n)$.
\\
So the initial condition $x_0$, if written as a column vector, has the form of
\[
x_0=\left(
\begin{array}{c}
x_0^+
\\
-\sum_{i=1}^{\infty}C(i-1,0)^{-1}\eta^-\left(i-1,x_{i-1}\right).
\end{array}
\right)
\]
Written as a column vecto, the solution of the dynamical system is of the form of
\[
\left(
\begin{array}{l}
x_{k+1}^+
\\
x_{k+1}^-
\end{array}
\right)
=\left(
\begin{array}{l}
B\left(k,0\right)x_0^++\sum_{i=0}^kB\left(k,i+1\right)\eta^+\left(i,x_i\right)
\\
C\left(k,0\right)x_0^-+\sum_{i=0}^kC\left(k,i+1\right)\eta^-\left(i,x_i\right).
\end{array}
\right)
\]
Plugging the equation (\ref{eq:inftysum}) back to the above expression, we have 
\begin{align}
x_{k+1}&=\left(
\begin{array}{c}
B\left(k,0\right)x_0^+
\\
-\sum_{i=0}^kC\left(k,i+1\right)\eta^-\left(i,x_i\right)-\sum_{i=0}^\infty C(k+1+i,k+1)^{-1}\eta^-\left(k+1+i,x_{k+1+i}\right)
\end{array}
\right)
\\
&+\left(
\begin{array}{c}
\sum_{i=0}^kB\left(k,i+1\right)\eta^+\left(i,x_i\right)
\\
\sum_{i=0}^kC\left(k,i+1\right)\eta^-\left(i,x_i\right)
\end{array}
\right)
\\
&=\left(
\begin{array}{c}
B\left(k,0\right)x_0^++\sum_{i=0}^kB\left(k,i+1\right)\eta^+\left(i,x_i\right)
\\
-\sum_{i=0}^\infty C(k+1+i,k+1)^{-1}\eta^{-}\left(k+1+i,x_{k+1+i}\right).
\end{array}
\right)
\end{align}
Denote $\mathbb{B}(\delta)\subset\mathbb{R}^d$ the ball around 0 with Euclidean radius $\delta$. Denote 
\[
\ell_0(\mathbb{B}(\delta))=\{\{u_n\}_{n\in\mathbb{N}}\subset\mathbb{B}(\delta):\lim_{n\rightarrow\infty}u_n=0\}
\] 
 the metric space of  sequences whose entries are in $\mathbb{B}(\delta)$, with metric defined as
\begin{equation}\label{metric}
d(u,v):=\sup_{n\ge 0}\{\norm{u_n-v_n}\}.
\end{equation}
for any $u=\{u_n\}_{n\in\mathbb{N}}$ and $v=\{v_n\}_{n\in\mathbb{N}}$ in the ball $\mathbb{B}(\delta)$.
Then $\ell_0(\mathbb{B}(\delta))$ is a complete metric space 
Reason is as follows:
\\
Let $u_1=\{u_{1j}\}_{j\in\mathbb{N}}$, $u_2=\{u_{2j}\}_{j\in\mathbb{N}}$,..., $u_i=\{u_{ij}\}_{j\in\mathbb{N}}$,... be a sequence of sequences in $\ell_0(\mathbb{B}(\delta))$. Suppose $\{u_i\}_{i\in\mathbb{N}}$ is Cauchy with respect to the metric defined by \ref{metric},i.e. given any $\epsilon>0$, there exists integer $L>0$, such that 
\[
d(u_n,u_m)=\sup_{j\ge 0}\{\norm{u_{nj}-u_{mj}}\}<\epsilon
\]
for all $n,m>L$.
This means that for each $j$, there exists a point $u_{*j}\in\mathbb{B}(\delta)$ such that $\lim_{i\rightarrow\infty}u_{ij}=u_{*j}$. And then we denote the limit sequence as 
  $u_{*}=\{u_{*j}\}_{j\in\mathbb{N}}$.  Furthermore, letting $m\rightarrow\infty$, we have that
  \[
  \norm{u_{nj}-u_{*j}}<\epsilon
  \]
  for all $n>L$. Fixing $n$ and letting $j\rightarrow\infty$, we have $u_{*j}\rightarrow 0$ since $u_{nj}\rightarrow 0$. And this shows that $u_*\in\ell_0(\mathbb{B}(\delta))$.

Define the operator $T$ for each sequence $x=\{x_n\}_{n\in\mathbb{N}}\subset \mathbb{R}^d$ to be
\begin{align}\label{def:T}
\left(Tx\right)_{k+1}=\left(
\begin{array}{c}
B\left(k,0\right)x_0^++\sum_{i=0}^kB\left(k,i+1\right)\eta^+\left(i,x_i\right)
\\
-\sum_{i=0}^\infty C(k+1+i,k+1)^{-1}\eta^{-}\left(k+1+i,x_{k+1+i}\right).
\end{array}
\right)
\end{align} 
for $k\ge 0$ and $\left(Tx\right)_0=x_0$.
\\
Next we prove that $T$ is a contraction map when choosing sequence in a small enough neighborhood around $0$.
\\
Take $\mathbb{B}(\delta)$ a small enough neighborhood around $0$ such that the Lipschitz condition is satisfied. Let $u=\{u_n\}_{n\in\mathbb{N}}\subset \mathbb{B}(\delta)$ and $v=\{v_n\}_{n\in\mathbb{N}}\subset \mathbb{B}(\delta)$. Then we have 
\begin{align}
\left(Tu-Tv\right)_{k+1}&=(Tu)_{k+1}-(Tv)_{k+1}
\\
&=\left(
\begin{array}{c}
B\left(k,0\right)u_0^++\sum_{i=0}^kB\left(k,i+1\right)\eta^+\left(i,u_i\right)
\\
-\sum_{i=0}^\infty C(k+1+i,k+1)^{-1}\eta^{-}\left(k+1+i,u_{k+1+i}\right)
\end{array}
\right)
\\
&-\left(
\begin{array}{c}
B\left(k,0\right)v_0^++\sum_{i=0}^kB\left(k,i+1\right)\eta^+\left(i,v_i\right)
\\
-\sum_{i=0}^\infty C(k+1+i,k+1)^{-1}\eta^{-}\left(k+1+i,v_{k+1+i}\right)
\end{array}
\right)
\\
&=\left(
\begin{array}{c}
B\left(k,0\right)(u_0^+-v_0^+)+\sum_{i=0}^kB\left(k,i+1\right)(\eta^+(i,u_i)-\eta^+\left(i,v_i\right))
\\
-\sum_{i=0}^\infty C(k+1+i,k+1)^{-1}(\eta^-(k+1+i,u_{k+1+i})-\eta^{-}\left(k+1+i,v_{k+1+i}\right)).
\end{array}
\right)
\end{align}
Use spectrum norm $\norm{\cdot}$ for matrices, we have
\begin{align}
\abs{(Tu-Tv)_{k+1}}&\le \norm{B(k,0)}\abs{u_0^+-v_0^+}+\sum_{i=0}^k\norm{B(k,i+1)}\norm{\eta^+(i,u_i)-\eta^+(i,v_i)}
\\
&+\sum_{i=0}^{\infty}\norm{C(k+1+i,k+1)^{-1}}\norm{\eta^-(k+1+i,u_{k+1+i})-\eta^-(k+1+i,v_{k+1+i})}
\\
&\text{(by Lipschitz assumption (\ref{Lipschitz}))}
\\
&\le\norm{B(k,0)}\norm{u_0^+-v_0^+}+\sum_{i=0}^k\norm{B(k,i+1)}\alpha_i\epsilon\norm{u_i-v_i}
\\
&+\sum_{i=0}^{\infty}\norm{C(k+1+i,k+1)^{-1}}\alpha_{k+1+i}\epsilon\norm{u_{k+1+i}-v_{k+1+i}}
\\
&\le\norm{B(k,0)}d(u,v)+\sum_{i=0}^k\norm{B(k,i+1)}\alpha_i\epsilon d(u,v)
\\
&+\sum_{i=0}^{\infty}\norm{C(k+1+i,k+1)^{-1}}\alpha_{k+1+i}\epsilon d(u,v)
\\
&=\norm{B(k,0)}d(u,v)+\sum_{i=0}^k\alpha_i\epsilon\norm{B(k,i+1)}d(u,v)
\\
&+\sum_{i=0}^{\infty}\alpha_{k+1+i}\epsilon\norm{C(k+1+i,k+1)^{-1}}d(u,v)
\\
&=\norm{B(k,0)}d(u,v)+\epsilon d(u,v)\left(\sum_{i=0}^k\alpha_i\norm{B(k,i+1)}\right)
\\
&+\epsilon d(u,v)\left(\sum_{i=0}^{\infty}\alpha_{k+1+i}\norm{C(k+1+i,k+1)^{-1}}\right)\label{eq:bounded}
\end{align}

Next we proceed to prove that 
\[
\norm{B(k,0)}+\epsilon\left(\sum_{i=0}^k\alpha_i\norm{B(k,i+1)}\right)+\epsilon\left(\sum_{i=0}^{\infty}\alpha_{k+1+i}\norm{C(k+1+i,k+1)^{-1}}\right)
\]
can be taken less than 1 so that $T$ is a contraction map on $\ell_0(\mathbb{B}(\delta))$.
\begin{lemma}\label{boundedC}
\[
R_k=\sum_{i=0}^{\infty}\alpha_{k+1+i}\norm{C(k+1+i,k+1)^{-1}}
\]
is a convergent series for each $k\in\mathbb{N}^+$. Moreover, there exists a constant $K_2>0$ such that $R_k\le K_2$ for all $k\in\mathbb{N}^+$.
\end{lemma}
\begin{proof} 
 Denote $\lambda$ the least negative eigenvalue, then the spectrum norm of $C(k+1+i,k+1)^{-1}$ is
\[
\norm{C(k+1+i,k+1)^{-1}}=\prod_{j=k+1}^{k+1+i}(1-\alpha_j\lambda)^{-1}.
\]
Since the sequence $\alpha_i$ is chosen to be small, we have
\begin{align}
R_k&=\sum_{i=0}^{\infty}\alpha_{k+1+i}\norm{C(k+1+i,k+1)^{-1}}
\\
&\le\alpha_k\sum_{i=0}^{\infty}\norm{C(k+1+i,k+1)^{-1}}
\\
&=\alpha_k\sum_{i=0}^{\infty}\prod_{j=k+1}^{k+1+i}(1-\alpha_j\lambda)^{-1}.
\end{align}
Using the inequality $1+x\le e^x$, we have
\[
(1-\alpha_j\lambda)^{-1}=\frac{1}{1-\alpha_j\lambda}=1+\frac{\alpha_j\lambda}{1-\alpha_j\lambda}\le \exp\left(\frac{\alpha_j\lambda}{1-\alpha_j\lambda}\right)
\]
and 
\[
\prod_{j=k+1}^{k+1+i}(1-\alpha_j\lambda)^{-1}\le\exp\left(\sum_{j=k+1}^{k+1+i}\frac{\alpha_j\lambda}{1-\alpha_j\lambda}\right),
\]
and thus 
\begin{equation}\label{boundR}
R_k\le\alpha_k \sum_{i=0}^{\infty}\exp\left(\sum_{j=k+1}^{k+1+i}\frac{\alpha_j\lambda}{1-\alpha_j\lambda}\right).
\end{equation}
Since by assumption, $\alpha_j\in\Omega\left( \frac{1}{j^p}\right)$ and $\lambda<0$, so $1-\alpha_j\lambda$ is positive and bounded, i.e. $1<1-\alpha_j\lambda<c$. And then the following inequalities hold:
\[
\frac{\alpha_j}{1-\alpha_j\lambda}\ge\frac{1}{1-\alpha_j\lambda}\cdot\frac{1}{j^p}\ge\frac{1}{c}\cdot\frac{1}{j^p}.
\]
Multiplying by the negative number $\lambda$, we have
\[
\frac{\alpha_j\lambda}{1-\alpha_j\lambda}\le\frac{\lambda}{c}\cdot\frac{1}{j^p}.
\]
Combining with the inequality \ref{boundR}, we obtain
\[
R_k\le \alpha_k\sum_{i=0}^{\infty}\exp\left(\sum_{j=k+1}^{k+1+i}\frac{\alpha_j\lambda}{1-\alpha_j\lambda}\right)\le\alpha_k\sum_{i=0}^{\infty}\exp\left(\frac{\lambda}{c}\sum_{j=k+1}^{k+1+i}\frac{1}{j^p}\right).
\]
By definition of definite integral, we notice that
\begin{align}
\sum_{j=k+1}^{k+1+i}\frac{1}{j^p}&>\int_{k+1}^{k+2+i}\frac{1}{t^p}dt
\\
&=\frac{1}{1-p}(k+2+i)^{1-p}-\frac{1}{1-p}(k+1)^{1-p}.
\end{align}
Since $\lambda<0$,
\[
\frac{\lambda}{c}\sum_{j=k+1}^{k+1+i}\frac{1}{j^p}<\frac{\lambda}{c(1-p)}(k+2+i)^{1-p}-\frac{\lambda}{c(1-p)}(k+1)^{1-p}
\]
so we have
\begin{align}
\exp\left(\frac{\lambda}{c}\sum_{j=k+1}^{k+1+i}\frac{1}{j^p}\right)&<\exp\left(\frac{\lambda}{c(1-p)}(k+2+i)^{1-p}-\frac{\lambda}{c(1-p)}(k+1)^{1-p}\right)
\\
&=\exp\left(\frac{\lambda}{c(1-p)}(k+2+i)^{1-p}\right)\cdot\exp\left(-\frac{\lambda}{c(1-p)}(k+1)^{1-p}\right).
\end{align}
So for each fixed $k$, we have that
\begin{align}
R_k&\le\alpha_k\sum_{i=0}^{\infty}\exp\left(\frac{\lambda}{c}\sum_{j=k+1}^{k+1+i}\frac{1}{j^p}\right)
\\
&<\alpha_k\exp\left(-\frac{\lambda}{c(1-p)}(k+1)^{1-p}\right)\cdot\sum_{i=0}^{\infty}\exp\left(\frac{\lambda}{c(1-p)}(k+2+i)^{1-p}\right).
\end{align}
The series 
\begin{equation}\label{series:exp}
\sum_{i=0}^{\infty}\exp\left(\frac{\lambda}{c(1-p)}(k+2+i)^{1-p}\right)
\end{equation}
 has the same convergence as the integral 
\[
\int_k^{\infty}\exp\left(-t^{1-p}\right)dt.
\]
Notice that 
\begin{align}
\int_k^{\infty}\exp\left(-t^{1-p}\right)dt&=\int_{k^{1-p}}^{\infty}\exp(-u)\frac{1}{1-p}u^{\frac{1}{1-p}-1}du
\\
&=\frac{1}{1-p}\int_{k^{1-p}}^{\infty}\exp(-u)u^{\frac{1}{1-p}-1}du
\\
&=\frac{1}{1-p}\Gamma\left(\frac{1}{1-p},k^{1-p}\right), (\text{the incomplete Gamma function})
\end{align}
which implies that $\int_k^{\infty}\exp(-t^{1-p})dt$ converges, so does the series \ref{series:exp}.
\\
Since the incomplete Gamma function $\Gamma(s,x)$ has the property
\[
\frac{\Gamma(s,x)}{x^{s-1}e^{-x}}\rightarrow 1\ \ \ \text{as}\ \ \ x\rightarrow\infty,
\]
let $s=\frac{1}{1-p}$ and $x=k^{1-p}$ so that $x^{s-1}=(k^{1-p})^{\frac{1}{1-p}-1}=k^p$, we have that
\[
\frac{1}{k^p}e^{k^{1-p}}\Gamma\left(\frac{1}{1-p},k^{1-p}\right)=\frac{\Gamma(\frac{1}{1-p},k^{1-p})}{k^pe^{-k^{1-p}}}\rightarrow 1
\]
as $k\rightarrow\infty$. This implies that $R_k$ is bounded as $k\rightarrow\infty$.
\end{proof}

\begin{lemma}\label{boundedB}
The sequence 
\[
S_k=\sum_{i=0}^k\alpha_i\norm{B(k,i+1)}
\]
is uniformly bounded for all $k\in\mathbb{N}$, i.e. there exists positive number $K_1$ such that $S_k\le K_1$
\end{lemma}
\begin{proof}  Since $B(k,i+1)$ is diagonal, denote $\lambda$ the least positive eigenvalue of $H$, by definition of $B(k,i+1)$, we have that
\[
\norm{B(k,i+1)}=(1-\alpha_k\lambda)\dotsm(1-\alpha_{i+1}\lambda).
\]
Then
\[
S_k=\alpha_0(1-\alpha_k\lambda)\dotsm(1-\alpha_1\lambda)+...+\alpha_k.
\]
Notice that
\[
S_{k+1}=(1-\alpha_{k+1}\lambda)S_k+\alpha_{k+1}.
\]
Consider the difference between $S_{k+1}$ and $S_k$, we have
\begin{align}
S_{k+1}-S_k&=(1-\alpha_{k+1}\lambda)S_k+\alpha_{k+1}-S_k
\\
&=S_k-\alpha_{k+1}\lambda S_k+\alpha_{k+1}-S_k
\\
&=\alpha_{k+1}(1-\lambda S_k).
\end{align}

We observe the following facts:
\begin{enumerate}
\item If $S_k=\frac{1}{\lambda}$, then $S_k=S_{k+1}\equiv\frac{1}{\lambda}$.
 \item If $S_k>\frac{1}{\lambda}$, or equivalently $1-\lambda S_k<0$, then 
$S_{k+1}-S_k<0$, and $S_k$ decreases until $S_{k_1}<\frac{1}{\lambda}$ for some $k_1\in\mathbb{N}$.
\item If $S_k<\frac{1}{\lambda}$, or equivalently $1-\lambda S_k>0$, then $S_{k+1}-S_k>0$, and $S_k$ increases until $S_k>\frac{1}{\lambda}$.
\end{enumerate}
So $S_k$ decreases or increases to $\frac{1}{\lambda}$ (meaning that $S_k$ is bounded), or $S_k$ oscillates around $\frac{1}{\lambda}$. Suppose that $S_k<\frac{1}{\lambda}$ and $S_{k+1}>\frac{1}{\lambda}$, we have that
\[
S_{k+1}\le S_k+\alpha_{k+1}\le \frac{1}{\lambda}+\frac{1}{\lambda}=\frac{2}{\lambda}
\]
when $k$ is large so that $\alpha_k<\frac{1}{\lambda}$. Then in conclusion, $S_k$ is bounded, and the proof completes.
\end{proof}
Next result shows that $T$ maps a sequence converging to 0 to another sequence converging to 0. And this is a prerequisite for $T$ to be a well defined map on the complete metric space $\ell_0(\mathbb{B}(\delta))$ to itself.
\begin{lemma}
Suppose $x=\{x_k\}_{k\in\mathbb{N}}$ and $\lim_{k\rightarrow\infty}x_k=0$. Then $\lim_{k\rightarrow\infty}(Tx)_{k+1}=0$
\end{lemma}
\begin{proof}
Denote $(Tx)_{k+1}^+$ and $(Tx)_{k+1}^-$ the stable and unstable component of $(Tx)_{k+1}$ respectively. We prove $\lim_{k\rightarrow\infty}(Tx)_{k+1}^+= 0$ and $\lim_{k\rightarrow \infty}(Tx)_{k+1}^-=0$ separately.
\\
\\
1. $\lim_{k\rightarrow\infty}(Tx)_{k+1}^+= 0$: 
\\
According to the definition of $T$ in \ref{def:T}, 
\[
(Tx)_{k+1}^+=B(k,0)x_0^++\sum_{i=0}^kB(k,i+1)\eta^+(i,x_i).
\]
 Since $\norm{B(k,0)}\rightarrow 0$ as $k\rightarrow\infty$, it is enough to show that 
 \[
 \sum_{i=0}^kB(k,i+1)\eta^+(i,x_i)\rightarrow 0
 \]
 as $k\rightarrow\infty$.
 From the Lipschitz condition on $\eta$, we have that
\begin{align}
\norm{\sum_{i=0}^k B(k,i+1)\eta^+(i,x_i(x_0))}&\le\epsilon\sum_{i=0}^k\norm{B(k,i+1)}\cdot\alpha_i\norm{x_i}
\\
&=\epsilon\left((1-\alpha_k\lambda)\dotsm(1-\alpha_{1}\lambda)\alpha_0\norm{x_0}+\dotsm(1-\alpha_{k}\lambda)\alpha_{k-1}\norm{x_{k-1}}+\alpha_{k}\norm{x_{k}}\right).
\end{align}
Denote the sum above as
\[
S_k=(1-\alpha_k\lambda)\dotsm(1-\alpha_{1}\lambda)\alpha_0\norm{x_0}+\dotsm(1-\alpha_{k}\lambda)\alpha_{k-1}\norm{x_{k-1}}+\alpha_{k}\norm{x_{k}}.
\]
Notice that 
\[
S_{k+1}=(1-\alpha_{k+1}\lambda)S_k+\alpha_{k+1}\norm{x_{k+1}},
\]
and then 
\[
S_{k+1}-S_k=\alpha_{k+1}(\norm{x_{k+1}}-\lambda S_k).
\]
 From the proof of Lemma \ref{boundedB}, we know that $S_k$ is bounded, and thus $\abs{S_{k+1}-S_k}\rightarrow 0$ as $k\rightarrow\infty$. Similar to proof of \ref{boundedB}, we have the following observation:
 \begin{enumerate}
 \item If $S_{k+1}-S_k>0$, then $\abs{x_{k+1}}-\lambda S_k>0$, or $S_k<\frac{\abs{x_k}}{\lambda}$;
 \item If $S_{k+1}-S_k<0$, then $\abs{x_{k+1}}-\lambda S_k<0$, or $S_k>\frac{\abs{x_k}}{\lambda}$;
 \item If $S_{k+1}-S_k=0$, then $S_k=\text{constant}$.
 \end{enumerate}  
 So the sequence $S_k$ is either 
 \begin{enumerate}
 \item decreasing but $S_k>\frac{\norm{x_k}}{\lambda}$, 
 \item oscillating around $\frac{\norm{x_k}}{\lambda}$.
 \end{enumerate}
 If $S_k$ is of case 1, then $\lim_kS_k$ exists. Suppose that this limit is positive, but since we have $\sum\alpha_k=\infty$ and
 \[
 S_{k+1}=S_k+\alpha_{k+1}(\norm{x_{k+1}}-\lambda S_k)
 \]
 implying that $S_k\rightarrow\infty$. So we conclude that $\lim_k S_k=0$, contradicting to the fact that $\lim_k S_k$ exists. So the $\lim_k S_k$ must be $0$ if $S_k$ is of case 1.
 \\
 If $S_k$ is of case 2, then immediately $\liminf S_k=0$. Suppose that $\limsup S_k>0$. Since $S_k$ decreases whenever $S_k>\frac{\norm{x_k}}{\lambda}$ and $S_k$ increases whenever $S_k<\frac{\norm{x_k}}{\lambda}$, we can find a subsequence $S_{k_m}$, with $S_{k_m-1}<S_{k_m}$, converging to $\limsup S_k$ as $m\rightarrow\infty$. But this is impossible since $S_{k_m-1}<\frac{\norm{x_k}}{\lambda}$ and then $S_{k_m-1}\rightarrow 0$ as $m\rightarrow\infty$, which means $\lim_m\abs{S_{k_m-1}-S_{k_m}}$ is positive, contradicting to the fact that $\lim_k\abs{S_k-S_{k+1}}=0$. And thus, we have $\limsup_k S_k=0$, meaning that $\lim_k S_k=0$.
 \\
 So we conclude that either in case 1 or 2, the limit $\lim_k S_k=0$, which completes the proof of part 1.
 \\
 \\
 2. $\lim_{k\rightarrow\infty}(Tx)_{k+1}^-=0$: 
 \\
 According to the equation \ref{def:T}, 
 \[
 (Tx)_{k+1}^-=-\sum_{i=0}^{\infty}C(k+1+i,k+1)^{-1}\eta^-(k+1+i,x_{k+1+i}).
 \]
 And from the Lipschitz condition of on $\eta$, we have that
 \begin{align}
\norm{(Tx)_{k+1}^-}&\le \sum_{i=0}^{\infty}\norm{C(k+1+i,k+1)^{-1}}\norm{\eta^{-}(k+1+i,x_{k+1+i})}
\\
&\le\sum_{i=0}^{\infty}\norm{C(k+1+i,k+1)^{-1}}\norm{\eta(k+1+i,x_{k+1+i})}
\\
&\le\sum_{i=0}^{\infty}\norm{C(k+1+i,k+1)^{-1}}\epsilon\alpha_{k+1+i}\norm{x_{k+1+i}}
\\
&\le\sum_{i=0}^{\infty}\norm{C(k+1+i,k+1)^{-1}}\epsilon\alpha_{k+1+i}\sup_{n> k}\norm{x_n}
\\
&\le\sup_{n> k}\abs{x_n}\cdot K_2 \ \ \ \text{(Lemma \ref{boundedC})}
\end{align}
Since $\{x_n\}$ converges to 0 as $n\rightarrow\infty$, $\sup_{n> k}\abs{x_n}\rightarrow 0$ as $k\rightarrow\infty$. And this completes the proof of part 2.
\end{proof}

\begin{lemma}\label{contraction}
There exists a real number $\delta>0$ such that the operator $T$ given by equation \ref{def:T}
\[
T:\ell_0(\mathbb{B}(\delta))\rightarrow\ell_0(\mathbb{B}(\delta))
\]
is a contraction map.
\end{lemma}
\begin{proof}
From Lemma \ref{boundedC} and Lemma \ref{boundedB}, we know that in equation (\ref{eq:bounded}), 

\begin{equation}\label{boundedK1}
\sum_{i=0}^k\alpha_i\norm{B(k,i+1)}\le K_1
\end{equation}
and
\begin{equation}\label{boundedK2}
\sum_{i=0}^{\infty}\alpha_{k+1+i}\norm{C(k+1+i,k+1)^{-1}}\le K_2.
\end{equation}
Since $B(k,0)$ is on the stable subspace and whose norm is calculated by
\[
\norm{B(k,0)}=\prod_{i=0}^k(1-\alpha_i\lambda)
\]
where $\lambda>0$, we have
\[
\norm{B(k,0)}\le\norm{B(0,0)}=1-\alpha_0\lambda<1.
\]
 Then we can choose small positive $\epsilon$ so that 
 \[
 \epsilon<\frac{\alpha_0\lambda}{K_1+K_2}.
 \]
 Define the constant $K$ to be
\begin{equation}\label{contractconst}
K:=1-\alpha_0\lambda+\epsilon\left(K_1+K_2\right),
\end{equation}
and by the choice of $\epsilon$, we know that $K<1$. Let $\delta>0$ be the radius corresponding to $\epsilon$ so that the Lipschitz condition is satisfied.
\\
Combining \ref{eq:bounded}, \ref{boundedK1} and \ref{boundedK2}, we have that
\[
\norm{(Tu-T_v)_{k+1}}\le \left(\norm{B(k,0)}+\epsilon(K_1+K_2)\right)d(u,v)\le Kd(u,v).
\]
Since above $k$ is taken arbitrarily, we conclude that
\[
\norm{Tu-Tv}\le Kd(u,v).
\]
So $T$ is a contraction map. 
\end{proof}
And since $\ell_0(\mathbb{B}(\delta))$ is a complete metric space, according to Banach fixed point theorem, there exists a unique sequence, denoted as $x=\{x_n\}_{n\in\mathbb{N}}$, such that 
\[
Tx=x
\]
with initial condition satisfying 
\begin{equation}\label{eq:initial}
(x_0^+,x_0^-)=(x_0^+,-\sum_{i=0}^\infty C(k+1+i,k+1)^{-1}\eta^{-}\left(k+1+i,x_{k+1+i}\right)).
\end{equation}
If we consider the sequence $x$ as a sequence of functions with the initial condition as the variable, the general term $x_n$ is written as $x_n(x_0)$, then the equation (\ref{eq:initial}) is written as
\[
(x_0^+,x_0^-)=\left(x_0^+,-\sum_{i=0}^\infty C(k+1+i,k+1)^{-1}\eta^{-}\left(k+1+i,x_{k+1+i}(x_0^+,x_0^-)\right)\right).
\]
This means that if the some initial condition $x_0$ goes to 0 through the discrete time process $\{x_n(x_0)\}$, its stable and unstable component must satisfy following relation:
\[
x_0^-=-\sum_{i=0}^\infty C(k+1+i,k+1)^{-1}\eta^{-}\left(k+1+i,x_{k+1+i}(x_0^+,x_0^-)\right).
\]
Denote
\[
\Phi(x_0^+,x_0^-)=-\sum_{i=0}^\infty C(k+1+i,k+1)^{-1}\eta^{-}\left(k+1+i,x_{k+1+i}(x_0^+,x_0^-)\right)
\]
and the equation $x_0^-=\Phi(x_0^+,x_0^-)$ defines an implicit function $x_0^-=\varphi(x_0^+)$ by the uniqueness of Banach fixed point. Next we will show that $\varphi$ is differentiable with respect to $x_0^+$. Since it is enough to show the function $\Phi(x_0^+,x_0^-)$
 is differentiable with respect to $x_0^+$. And it is enough to show that each $x_n(x_0)$, if considered as a function of initial condition $x_0$, is differentiable with respect to $x_0^+$.
\begin{lemma}
The solution $x_n(x_0^+,x_0^-)$ is of $C^1$ with respect to $x_0^+$ provided $\eta(n,x)$ is of $C^1$. 
\end{lemma}
\emph{Proof}. It is equivalent to show that $\frac{\partial x_n}{\partial x_{0,j}}$, $j=1,..,d$, where $d$ is the dimension of stable vector space, exist and are continuous for small $\abs{x_0}$.
\\
Let $P^+$ and $P^-$ be the projection operators to the stable and unstable subspaces respectively, then the solution (with initial condition $x_0$) of the dynamical system can be written as
\begin{align}\label{eq:general}
x_{k+1}(x_0)= &A(k,0)P^+x_0+\sum_{i=0}^kA(k,i+1)P^+\eta(i,x_i(x_0))
\\
&-\sum_{i=0}^{\infty}A(k+1+i,k+1)^{-1}P^-\eta(k+1+i,x_{k+1+i}(x_0)).
\end{align}
Let $h$ be a scalar and $e_j$ be the $j$th standard basis. Denote
\[
q(n,x_0,h)=\frac{x_{n}(x_0+he_j)-x_{n}(x_0)}{h}.
\]
Notice the following identity holds:
\begin{align}
\frac{\eta(n,x_n(x_0+he_j))-\eta(n,x_n(x_0))}{h}=&\frac{\eta(n,x_n(x_0+he_j))-\eta(n,x_n(x_0))}{h}
\\
&+D\eta(n,x_n(x_0))q(n,x_0,h)
\\
&-D\eta(n,x_n(x_0))q(n,x_0,h).
\end{align}
Plugging above identity to \ref{eq:general}, we can compute the difference quotient $q(k+1,x_0,h)=\frac{x_{k+1}(x_0+he_j)-x_{k+1}(x_0)}{h}$
\begin{align}
q(k+1,x_0,h)&=A(k,0)P^+\left(\frac{(x_0+he_j)-x_0}{h}\right)
\\
&+\sum_{i=0}^kA(k,i+1)P^+\left(\frac{\eta(i,x_i(x_0+he_j))-\eta(i,x_i(x_0))}{h}\right)
\\
&-\sum_{i=0}^{\infty}A(k+1+i,k+1)^{-1}P^-\left(\frac{\eta(k+1+i,x_{k+1+i}(x_0+he_j))-\eta(k+1+i,x_{k+1+i}(x_0))}{h}\right)
\\
&=A(k,0)P^+e_j+\sum_{i=0}^kA(k,i+1)P^+\left(D\eta(i,x_i(x_0))q(i,x_0,h)+\Delta_i\right)
\\
&-\sum_{i=0}^{\infty}A(k+1+i,k+1)^{-1}P^-\left(D\eta(k+1+i,x_{k+1+i}(x_0))q(k+1+i,x_0,h)+\Delta_{k+1+i}\right),
\end{align}
where 
\[
\Delta_n=\frac{\eta(n,x_n(x_0+he_j))-\eta(n,x_n(x_0))}{h}-D\eta(n,x_n(x_0))q(n,x_0,h).
\]
Since for the solution $x(n,x_0)$, $x(n,x_0)\rightarrow 0$ as $n\rightarrow\infty$,
and $\norm{\eta(n,x)-\eta(n,\bar{x})}\le\epsilon\norm{x-\bar{x}}$, we have that
\[
\norm{D\eta(n,x_n(x_0))}\le\epsilon d.
\]
Given $\delta>0$, $\abs{h}$ can be chosen small that by the mean value theorem and the continuity of $D\eta$, we have 
\begin{align}
\norm{\Delta_n}&\le\frac{1}{h}\norm{D\eta(n,x')}\cdot\norm{x_n(x_0+he_j)-x_n(x_0)}+\norm{D\eta(n,x_n(x_0))}\cdot\norm{q(n,x_0,h)}
\\
&=\left(\norm{D\eta(n,x_n')}+\norm{D\eta(n,x_n(x_0))}\right)\cdot\norm{q(n,x_0,h)}
\\
&\le\delta\norm{q(n,x_0,h)},
\end{align}
where $x_n'$ is a point on the line segment joining $x_n(x_0+he_j)$ and $x_n(x_0)$.
Since 
\[
\frac{\norm{\eta(i,x_i(x_0+he_j))-\eta(i,x_i(x_0))}}{\abs{h}}\le\alpha_i\epsilon
\]
so $\norm{q(n,x_0,h)}$ is bounded, denoted as
\[
\norm{q(n,x_0,h)}\le M.
\]
And then $\norm{\Delta_n}\le \delta M$.
Define the operator as
\begin{align}
\psi(k+1,x_0)=&A(k,0)P^+e_j+\sum_{i=0}^kA(k,i+1)P^+D\eta(i,x_i(x_0))\psi(i,x_0)
\\
&-\sum_{i=0}^{\infty}A(k+1+i,k+1)^{-1}P^-D\eta(k+1+i,x_{k+1+i}(x_0))\psi(k+1+i,x_0).
\end{align}
Consider the difference
\begin{align}
&q-\psi
\\
&=\sum_{i=0}^kA(k,i+1)P^+\left(D\eta(i,x_i(x_0))\left(q(i,x_0,h)-\psi(i,x_0)\right)+\Delta_i\right)
\\
&-\sum_{i=0}^{\infty}A(k+1+i,k+1)^{-1}P^-\left(D\eta(k+1+i,x_{k+1+i}(x_0))\left(q(k+1+i,x_0,h)
-\psi(k+1+i,x_0)\right)+\Delta_{k+1+i}\right).
\end{align}
Notice that the part of infinite sum converges to 0 as $k\rightarrow\infty$, one can choose $k$ large enough so that the norm of the infinite sum to be small, and then we have for any small $\epsilon'>0$, the $\sup\norm{q-\psi}$ satisfies
\begin{align}
\sup\norm{q-\psi}&\le\sum_{i=0}^k\norm{A(k,i+1)P^+\left(D\eta(i,x_i(x_0))\left(q(i,x_0,h)-\psi(i,x_0)\right)+\Delta_i\right)}+\epsilon'
\\
&\le K'\epsilon d\sup\norm{q-\psi}+ K''\delta.
\end{align}
Where $K''$ is the bound from that $|\Delta_i|\rightarrow 0$ as $i\rightarrow\infty$.
One choose neighborhood small enough so that $K'\epsilon d<\frac{1}{2}$ and then we have
\[
\sup\norm{q-\psi}<K''\delta.
\]
Since $\delta\rightarrow 0$ as $h\rightarrow 0$, so $\sup \norm{q-\psi}\rightarrow 0$ as $h\rightarrow 0$. And this means that the partial derivative $\frac{\partial x_n}{\partial \xi_i}$ exists and equals to $\psi$. 
\\
In the end we prove that $\bigcap_{k=0}^{\infty}\tilde{g}^{-1}(k,0,U)\subset V(0)$ and this can be done by contradiction. Assume that there is an initial point $x_0$ not in $V(0)$ that generates a sequence $\{x_k\}_{k\in\mathbb{N}}$ such that $x_k\in U$ as $k\rightarrow\infty$. Since $x_{k+1}^-=C_{k+1}x_k^-+\eta^-(k,x_k)$, we have that
\[
\norm{x_{k+1}^-}=\norm{C_{k+1}x_k^-+\eta^-(k,x_k)}\ge\abs{\norm{C_{k+1}x_k^-}-\norm{\eta^-(k,x_k)}}.
\]
Since $\norm{\eta^-(k,x_k)}\rightarrow 0$ as $k\rightarrow\infty$ due to $\alpha_k$, and $\norm{C_{k+1}x_k^-}\rightarrow\infty$ as $k\rightarrow\infty$ by assumption that $x_k^-$ is bounded away from $0$. But this contradicts to the assumption $x_{k+1}^-$ is bounded in $U$. The proof completes.
\end{proof}

\end{document}